\documentclass[leqno]{article}

\usepackage[OT1]{fontenc}
\usepackage{amsthm,amsmath,graphicx,amssymb,enumerate}
\usepackage[colorlinks,citecolor=blue,linkcolor=red,urlcolor=blue]{hyperref}
\usepackage{hypernat}

\numberwithin{equation}{section}

\newtheorem{Thm}{Theorem}[section]
\newtheorem{Prop}[Thm]{Proposition}
\newtheorem{Lem}[Thm]{Lemma}
\newtheorem{Cor}[Thm]{Corollary}
\newtheorem{Rem}[Thm]{Remark}
\newtheorem{Rems}[Thm]{Remarks}
\newtheorem{Ex}[Thm]{Example}
\newtheorem{Notation}[Thm]{Notation}
\newtheorem{Properties}[Thm]{Properties}

\newcommand{\cF}{({\cal{F}}_t)}
\newcommand{\cG}{({\cal{G}}_t)}
\newcommand{\bbP}{\mathbb P}
\newcommand{\bbR}{\mathbb R}
\newcommand{\bbE}{\mathbb E}
\newcommand{\bbV}{\mathbb V}
\newcommand{\bbD}{\mathbb D}
\newcommand{\bbL}{{\mathbb L}}
\newcommand{\bbS}{{\mathbb S}}
\newcommand{\bI}{{\bf I}}
\newcommand{\bD}{{\bf D}}
\newcommand{\bddL}{{\bf b}{\mathbb L}}
\newcommand{\bddP}{{\bf b}{\mathcal{P}}}

\newcommand{\HH}{{\cal H}^2}
\newcommand{\ttimes}{\hspace{-0.6mm}\times\hspace{-0.6mm}}
\newcommand{\synch}{\sim_{\textrm{synch}}}
\newcommand{\DtoE}{$[D\longmapsto E]$}
\newcommand{\EtoD}{$[D$ \reflectbox{$\longmapsto$} $E]$} 
\newcommand{\StoT}{$[S\longmapsto T]$}
\newcommand{\TtoS}{$[S$ \reflectbox{$\longmapsto$} $T]$} 
\newcommand{\DttoE}{$[[D\longmapsto E]]$}
\newcommand{\EttoD}{$[[D$ \reflectbox{$\longmapsto$} $E]]$} 
\newcommand{\SttoT}{$[[S\longmapsto T]]$}
\newcommand{\TttoS}{$[[S$ \reflectbox{$\longmapsto$} $T]]$} 
\newcommand{\DEF}{\mathrel{\mathrel{\mathop:}=}}

\makeatletter
\def\section{\@startsection {section}{1}{\z@}{-3.5ex plus -1ex minus 
-.2ex}{2.3ex plus .2ex}{\normalsize\bf}}
\makeatother

\begin{document}

\title{\textbf{\large Stochastic Calculus for a Time-changed Semimartingale and the Associated Stochastic Differential Equations}}
\author{\textsc{Kei Kobayashi}\thanks{Department of Mathematics, Tufts University, 503 Boston Avenue, Medford, MA 02155, USA. \textit{Email:} kei.kobayashi@tufts.edu}
        \vspace{1mm} \\ \textit{\normalsize Tufts University}}
\date{ }
\maketitle                                                       
\vspace{-2mm}
\renewcommand{\thefootnote}{\fnsymbol{footnote}}

\begin{abstract}
     It is shown that under a certain condition on a semimartingale and a time-change, 
     any stochastic integral driven by the time-changed semimartingale is a time-changed 
     stochastic integral driven by the original semimartingale.   
     As a direct consequence, a specialized form of the It\^o formula is derived. 
     When a standard Brownian motion is the original semimartingale, 
     classical It\^o stochastic differential equations driven by the Brownian motion with drift 
     extend to a larger class of stochastic differential equations involving a time-change with continuous paths.  
     A form of the general solution of linear equations in this new class is established, followed by consideration of some examples 
     analogous to the classical equations.  
     Through these examples, each coefficient of the stochastic differential equations in the new class is given meaning. 
     The new feature is the coexistence of a usual drift term along with
     a term related to the time-change.
\end{abstract}

\section{Introduction} \label{sec introduction}

     Among the most important results in the theory of stochastic integration is 
     the celebrated It\^o formula, which establishes a stochastic calculus for stochastic integrals driven by a semimartingale. 
     In general, given a $d$-dimensional semimartingale $X=(X^1,\ldots ,X^d)$ starting at $0$, 
     if $f:\bbR^d \longrightarrow \bbR$ is a $C^2$ function, 
     then $f(X)$ is a one-dimensional semimartingale, and, for all $t\geq 0$, with probability one
     \begin{align}  
          f(X_t)-f(0)\label{ItoFormula1} 
           &=\sum_{i=1}^d \int_0^t \dfrac{\partial f}{\partial x^i}(X_{s-}) dX_s^{i}	
             + \dfrac 12 \sum_{i,j=1}^d \int_0^t \dfrac{\partial^2 f}{\partial x^i \partial x^j}(X_{s-}) d[X^i,X^j]_s^c \\
           & \ \ \ +\sum_{0<s\leq t}\Bigl\{ f(X_s)-f(X_{s-})-\sum_{i=1}^d \dfrac{\partial f}{\partial x^i}(X_{s-})
              \Delta X_s^i \Bigr\}. \notag
     \end{align}
     One useful implication of the It\^o formula \eqref{ItoFormula1} is the product rule.  Namely, if $Y$ and $Z$ are both one-dimensional 
     semimartingales starting at $0$, then, for all $t\geq 0$, with probability one
     \begin{align}  
          Y_tZ_t=\int_0^t Y_{s-} dZ_s+\int_0^t Z_{s-} dY_s + [Y,Z]_t.  \label{ItoFormula2}
     \end{align} 
     These formulas are indispensable tools for working with stochastic differential equations (SDEs).  
     
     Our motivation to investigate stochastic integrals driven by a time-changed semimartingale originated in a desire 
     to develop a stochastic calculus when the time-change is the 
     first hitting time process of a stable subordinator of index between $0$ and $1$.
     Meerschaert and Scheffler~\cite{M-S_1,M-S_2} show that 
     this type of process arises as the scaling limit of continuous time random walks. 
     If the original semimartingale is a standard Brownian motion, then it is known that 
     the transition probability density of the time-changed Brownian motion satisfies a time-fractional partial differential equation (PDE). 
     However, a general PDE satisfied by the transition probability density of a solution to an SDE 
     which includes a term driven by the time-changed Brownian motion has not been completely revealed. 
     The stochastic calculus developed in this paper gives a way to deal with this problem.         
  
     Section \ref{sec preliminaries} first introduces the significant concept of synchronization, 
     which connects a semimartingale with a time-change in a certain manner.  
     A time-change $(T_t)$ is a c\`adl\`ag, nondecreasing family of stopping times. 
     Given a one-dimensional semimartingale $Z$ starting at $0$, 
     the composition of $Z$ and $T$, denoted $Z\circ T$ or $(Z_{T_t})$, is called the time-changed semimartingale.
     We occasionally refer to $t$ and $T_t$ as the original clock and the new clock, respectively. 
     With the notion of synchronization, Jacod~\cite{Jacod} explains how to recognize a time-changed stochastic integral 
     of the form $\int_0^{T_t} H_s dZ_s$
     in terms of an integral with respect to the time-changed semimartingale $(Z_{T_t})$ (Lemma \ref{Lem JACOD2}).
     However, this statement does not answer the following question: 
     \begin{itemize}
         \item[Q:] When and how can a stochastic integral $\int_0^t K_s dZ_{T_s}$ 
               driven by a time-changed semimartingale be realized by way of
               an integral driven by the original semimartingale $(Z_t)$?
     \end{itemize}
     In Section \ref{sec results}, Theorem \ref{Thm COV} provides a complete answer to the above question.  
     Namely, $\int_0^t K_s dZ_{T_s}=\int_0^{T_t} K_{S(s-)} dZ_s$, where $S$ is the first hitting time process of $T$. 
     An important corollary of Theorem \ref{Thm COV} is a form of the It\^o formula \eqref{ItoFormula1} 
     for a $C^2$ function of a process which contains a stochastic integral driven by a time-changed semimartingale $(Z_{E_t})$ 
     where $(E_t)$ is a continuous time-change, meaning a time-change with continuous paths.  
     The formula can be reexpressed in terms of usual stochastic integrals driven by 
     the original semimartingale 
     and the continuous part of the semimartingale's quadratic variation. 
     A generalization of this formula is a time-changed It\^o formula provided in Theorem \ref{Thm ITO}.

     Theorem \ref{Thm COV}, from which the time-changed It\^o formula is derived, 
     can be regarded as a powerful tool 
     in handling a new class of SDEs which are driven by 
     Lebesgue measure, a continuous time-change, and a time-changed semimartingale (Section \ref{sec LSDE}). 
     The simplest, yet quite significant subclass, of such SDEs are ones with linear coefficients: 
     \begin{align}
          dX_t&=\bigl(\rho_1(t,E_t)+\rho_2(t,E_t)X_t\bigr)dt +\bigl(\mu_1(t,E_t)+\mu_2(t,E_t)X_t\bigr)dE_t \ \ \ \ \ \label{INTRO1} \\
                              & \ \ \ \ +\bigl(\sigma_1(t,E_t)+\sigma_2(t,E_t)X_t\bigr)dB_{E_t},  \notag
     \end{align}
     where $B$ is a standard Brownian motion.  
     The new feature of this class of SDEs is the coexistence of a term representing a drift under the new clock $E_t$ 
     along with a usual drift based on the original clock $t$.
     Theorem \ref{Thm LSDE-3} establishes a general form of the solution to SDE \eqref{INTRO1}, in which again 
     Theorem \ref{Thm COV} is applied to obtain another representation of the solution.    
     
     Section \ref{sec example} compares some SDEs of the form \eqref{INTRO1} with classical It\^o SDEs, described as
     \begin{align}
          dY_t=\bigl(b_1(t)+b_2(t)Y_t\bigr)dt+\bigl(\tau_1(t)+\tau_2(t)Y_t\bigr)dB_t. \label{INTRO2}
     \end{align}  
     The comparison reveals the role of the $dE_t$ term appearing in SDE \eqref{INTRO1}. 
     Namely, $\mu_j$ can be ascribed to either $b_j$ or $\tau_j$ in \eqref{INTRO2}, 
     depending on the way the model \eqref{INTRO1} is constructed (Remark \ref{Rems LSDE-solution} (b)).  
     These examples also illustrate methods for obtaining statistical data of the solution, such as the mean and variance.

\section{Preliminaries --- Stochastic Integrals and Time-changes} \label{sec preliminaries} 
     Throughout this paper, a complete filtered probability space $(\Omega, {{\cal F}}, \cF, \bbP)$ is fixed, 
     where the filtration $\cF$ satisfies the \textit{usual conditions}; 
     that is, it is right-continuous and contains all the $\bbP$-null sets in $\cal F$.  
     For simplicity, unless mentioned otherwise, \emph{all processes
     are assumed to take values in $\bbR$ and start at $0$}. 

     A process $Z$ is said to be \textit{c\`adl\`ag} (resp.~\textit{c\`agl\`ad}) if $Z$ has right-continuous sample paths with left limits 
     (resp.~left-continuous sample paths with right limits).  
     The assumption that $Z$ is c\`adl\`ag or c\`agl\`ad requires the process to have at most countably many jumps.   
     Associated to a c\`adl\`ag process $Z$ is its jump process $\Delta Z=(\Delta Z_t)$ where $\Delta Z_t\DEF  Z_t-Z_{t-}$ 
     with $Z_{t-}$ denoting the left limit at $t$ and $Z_{0-}=0$ by convention.   
     Let $\bbD\cF$ and $\bbL\cF$ respectively denote the class of c\`adl\`ag, $\cF$-adapted processes 
     and that of c\`agl\`ad, $\cF$-adapted processes.
     
     A c\`adl\`ag process $Z$ is an $\cF$-\textit{semimartingale} if there exist an $\cF$-local martingale $M$ and 
     an $\cF$-adapted process $A$ of finite variation on compact sets such that $Z=M+A$.  
     Although this decomposition is not unique in general, the local martingale part $M$ can be uniquely decomposed 
     as $M=M^c+M^d$ with a continuous local martingale $M^c$ and a purely discontinuous local martingale $M^d$.  
     The process $M^c$ is determined independently of the initial decomposition of $Z$ into $M$ and $A$, and 
     $Z^c$ is defined to be the unique continuous local martingale part $M^c$ of $Z$; i.e., $Z^c\DEF M^c$  
     (\cite[I.\ Prop.\ 4.27]{J-S}).
     
     The class of semimartingales forms a real vector space which is closed under multiplication.  
     It is known to be the largest class of processes for which the It\^o-type stochastic integrals are defined.  
     Let $\mathcal{P}\cF$ be the smallest $\sigma$-algebra on $\bbR_+ \ttimes \Omega$ 
     which makes all processes in $\bbL\cF$ measurable.  
     Given an $\cF$-semimartingale $Z$, ${L(Z,{\cal F}_t)}$ denotes the class of ${\cal P}\cF$-measurable, 
     or \textit{$\cF$-predictable} processes $H$ for which 
     a stochastic integral driven by $Z$, denoted  
     $ (H\bullet Z)_t=\int_0^t H_s dZ_s$, can be constructed.      
     A brief summary of the construction appears in Appendix.   
      
     The \textit{quadratic variation} of a semimartingale $Z$, denoted $[Z,Z]$, can be defined by way of a stochastic integral.  
     It is the c\`adl\`ag, $\cF$-adapted, nondecreasing process given by 
     \begin{align}
          [Z,Z]_t\DEF Z_t^2-2\int_0^t Z_{s-}dZ_s. \label{QUAD} 
     \end{align}
     By polarization, the map $[\cdot,\cdot]$ becomes a symmetric, bilinear form on the class of semimartingales.  
     For semimartingales $Y$ and $Z$, note that $[Y,Z]^c$ does not denote its continuous martingale part, 
     which is of course zero, but it is defined to be its continuous part; namely,  
     \begin{align*}
          [Y,Z]_t^c\DEF [Y,Z]_t-\sum_{0<s\leq t}\Delta[Y,Z]_s =[Y,Z]_t-\sum_{0<s\leq t}\Delta Y_s\cdot \Delta Z_s.
     \end{align*}
     It follows by comparing this definition with Theorem 4.52 in \cite[Chap. I]{J-S} that $[Y,Z]^c=[Y^c,Z^c]$.      
     
     The following are some of the basic but key properties of stochastic integrals which will be employed in the subsequent sections.  

\begin{Properties}\label{Properties} \par
     Let $Y$ and $Z$ be $\cF$-semimartingales.  Let $H\in {L(Z,{\cal F}_t)}$.  
     \begin{enumerate}[(1)]
          \item $H\bullet Z$ is again an $\cF$-semimartingale.   \vspace{-4pt} \label{*1}
          \item $\Delta(H\bullet Z)=H\cdot \Delta Z$. \vspace{-4pt} \label{*2}
          \item Additionally, if $H\in L(Y,{{\cal F}_t})$, then $H\bullet (Z+Y)=H\bullet Z+H\bullet Y$. \vspace{-4pt}\label{*3}
          \item If $J\in L(H\bullet Z,{{\cal F}_t})$, 
                                   then $J\bullet (H\bullet Z)=(J\cdot H)\bullet Z$. \vspace{-4pt}\label{*4}
          \item If $K \in L(Y,{{\cal F}_t})$, 
                                   then $[H\bullet Z, K\bullet Y]=(H\cdot K)\bullet [Z,Y]$.\vspace{-4pt}\label{*5}
      \end{enumerate}
\end{Properties} 

     An $\cF$-\textit{time-change} is a c\`adl\`ag, nondecreasing family of $\cF$-stopping times.  
     It is said to be \textit{finite} if each stopping time is finite almost surely. 
     Let $(T_t)$ be a finite $\cF$-time-change and define a new filtration $\cG$ by ${\cal{G}}_t={\cal{F}}_{T_t}$.  
     Then $\cG$ also satisfies the usual conditions since the right-continuity of $\cF$ and $(T_t)$ implies that of $\cG$. 
     In addition, for any $\cF$-adapted process $Z$, the time-changed process $(Z_{T_t})$ is known to be $\cG$-adapted.  
     In fact, more can be said.  

\begin{Lem}\label{Lem JACOD1}\par \hspace{-2mm} 
     \textbf{\emph{(\cite[Cor.\ 10.12]{Jacod})}} \ 
     Let $Z$ be an $\cF$-semimartingale.  Let $(T_t)$ be a finite $\cF$-time-change.
     Then $(Z_{T_t})$ is a $\cG$-semimartingale where ${\cal G}_t\DEF {\cal F}_{T_t}$.  
\end{Lem}

On the other hand, 
the local martingale property is not always preserved under a finite
time-change.
     A simple example is a standard $\cF$-Brownian motion $Z=B$ 
     with the finite $\cF$-time-change $(T_t)$ defined by $T_t\DEF \inf\{s>0; B_s=t \}$.  
     In this case, $B_{T_t}=t$ for every $t\geq 0$.  Thus, the time-changed Brownian motion is no longer a local martingale. 

     One way to exclude this unexpected possibility is to introduce the notion of synchronization, 
     which turns out to be an essential concept in developing a stochastic calculus for integrals driven by a time-changed semimartingale.      
     A process $Z$ is said to be \textit{in synchronization with the time-change} $(T_t)$ if 
     $Z$ is constant on every interval $[T_{t-},T_t]$ almost surely.  
     We occasionally write $Z\synch T$ for shorthand.  
     Other properties that a time-change preserves appear in \cite[Thm.\ 10.16]{Jacod}. 
     In the literature, Jacod~\cite{Jacod}, Kallsen and Shiryaev~\cite{Kallsen-S} use the expression $(T_t)$-\textit{adapted}
     in describing a process being in synchronization with a time-change $(T_t)$.  
     A different terminology $(T_t)$-\textit{continuous} is used by Revuz and Yor~\cite{R-Y}.   
     Nevertheless, the term synchronization is adopted here to avoid any possible confusions or misunderstandings 
     that the other expressions may create.  
  
     One quite simple yet significant observation, which connects the notion of synchronization with stochastic integrals,
     is that if an $\cF$-semimartingale $Z$ is in synchronization with a finite $\cF$-time-change $(T_t)$ and if $H\in {L(Z,{\cal F}_t)}$, 
     then $\bigl(H_{T(t-)}\bigr)\in {L(Z\circ T,{\cal G}_t)}$, where, $H_{T(t-)}$ denotes the process $H$ evaluated 
     at the left limit point $T_{t-}$ of $T$ at $t$. 
     This observation leads to the consideration of two integral processes $(\int_0^t H_s dZ_s)$ and $(\int_0^t H_{T(s-)} dZ_{T_s})$.  
     By Property \ref{Properties} (\ref{*1}), these are semimartingales with respect to the filtrations $\cF$ and $\cG$, respectively.  
     By Lemma \ref{Lem JACOD1}, the former stochastic integral can be time-changed by $(T_t)$ to produce another $\cG$-semimartingale.   
     Jacod~\cite{Jacod} shows that the two $\cG$-semimartingales 
     $(\int_0^{T_t} H_s dZ_s)$ and $(\int_0^t H_{T(s-)} dZ_{T_s})$
     coincide for any $H\in L(Z,{\cal F}_t)$. 
     This fact plays a significant role in establishing the basic Theorem \ref{Thm COV}; 
     hence, it is stated here as a lemma.  

\begin{Lem}\label{Lem JACOD2}\par \hspace{-2mm}
     \textbf{\emph{(1st Change-of-Variable Formula \cite[Prop.\ 10.21]{Jacod})}} \
     Let $Z$ be an $\cF$-semimartingale which is in synchronization with a finite $\cF$-time-change $(T_t)$.
     If $H \in {L(Z,{\cal F}_t)}$, then $\bigl(H_{T(t-)}\bigr)\in {L(Z\circ T,{\cal G}_t)}$ where ${\cal G}_t\DEF {\cal F}_{T_t}$.  
     Moreover, with probability one, for all $t\geq 0$,
     \begin{align}
          \int_0^{T_t} H_s dZ_s =\int_0^t H_{T(s-)} dZ_{T_s}. \label{JACOD2}
     \end{align}
\end{Lem}

\begin{Lem}\label{Lem JACOD3}\par \hspace{-2mm}
     \textbf{\emph{(\cite[Thm.\ 10.17]{Jacod})}} \ 
     Let $Z$ be an $\cF$-semimartingale which is in synchronization with a finite $\cF$-time-change $(T_t)$.
     Then $Z^c$ and $[Z, Z]$ are also in synchronization with $(T_t)$.  Moreover, 
     \begin{align}
          [Z\circ T, Z\circ T]=[Z, Z]\circ T, \hspace{3mm} (Z\circ T)^c=Z^c\circ T. \label{JACOD3}
     \end{align}
\end{Lem}

     The following simple example explains the significance of the synchronization assumption in Lemmas 
     \ref{Lem JACOD2} and \ref{Lem JACOD3}. 

\begin{Ex} \par
\begin{em}
     Let $Z=B$ be a standard $\cF$-Brownian motion, and define a deterministic time-change $(T_t)$ by $T_t\DEF \bI_{[1,\infty)}(t)$, 
     where $\bI_\Lambda$ denotes the indicator function over a set $\Lambda$.
     Let $H$ be a deterministic process given by $H_t=\bI_{(1/2,\;\infty)}(t)$, then $H_{T(t-)}=\bI_{(1,\infty)}(t)$.  Hence, 
     \begin{align*} 
          &\int_0^{T_1} H_s dB_s =\int_0^1 H_s dB_s =\int_{1/2}^{1} dB_s =B_{1}-B_{1/2}\, ; \\
          &\int_0^1 H_{T(s-)}dB_{T_s} =\int_0^1 0 \;dB_{T_s}=0.
     \end{align*}
     Therefore, the two integrals in \eqref{JACOD2} fail to coincide.   
     Moreover, it follows from \eqref{QUAD} that  
     \begin{align*}
          [B\circ T, B\circ T]_1=(B_{T_1})^2-2\int_0^1 B_{T_{s-}}dB_{T_s}=B_1^2-2\int_0^1 0 \;dB_{T_s}=B_1^2, 
     \end{align*}
     whereas the fact that $[B,B]_t=t$ yields $([B, B]\circ T)_1=T_1=1$. 
     Therefore, the first equality in \eqref{JACOD3} does not hold.  
     Furthermore, since $B\circ T$ is not a continuous process, $(B\circ T)^c$ and $B^c \circ T(=B\circ T)$ fail to coincide.  
     Thus, the second equality in \eqref{JACOD3} does not hold either.      
     Note that the Brownian motion $B$ never stays flat on any time interval, and hence
     is not in synchronization with the above time-change $(T_t)$.  \qed
\end{em}
\end{Ex}

     The next lemma will be used in the proof of Theorem \ref{Thm COV}.

\begin{Lem}\label{Lem SYNCH}\par 
     Let $Z$ be an $\cF$-semimartingale which is in synchronization with a finite $\cF$-time-change $(T_t)$.
     Let $H\in{L(Z,{\cal F}_t)}$.  Then the stochastic integral $H\bullet Z$ is also in synchronization with $(T_t)$.  
\end{Lem}

\begin{proof}
     Fix $t\geq 0$, and let $u\in[T_{t-},T_t]$.  Since $Z\synch T$, $Z$ is constant on $[u, T_t]$; hence, 
     $(H\bullet Z)_{T_t}-(H\bullet Z)_u=\int_{u+}^{T_t}H_s dZ_s=0$. 
     Therefore, $(H\bullet Z)_{T_t}=(H\bullet Z)_u$.  Thus, $H\bullet Z$ is constant on $[T_{t-},T_t]$.       
\end{proof}

     The following lemma and its corollary clarify the situation of main concern in this paper.  
     The \textit{first hitting time process}, 
or the \textit{generalized inverse}, 
of a given c\`adl\`ag, nondecreasing process $S$ is a process $T$
     defined by $T_t\DEF \inf\{ u> 0\hspace{1pt};\hspace{1pt} S_u> t\}$.  It is easy to see that $T$ is also c\`adl\`ag and nondecreasing.  
     Note that every $\cF$-adapted, c\`adl\`ag, nondecreasing process has paths of finite variation on compact sets; 
     hence, \textit{a priori} it is an $\cF$-semimartingale.   
 
\begin{Lem} \label{Lem SETTING} \mbox{ } \par
     \begin{enumerate}[(1)]
          \item Let $S$ be a nondecreasing $\cF$-semimartingale such that $\lim_{t\to\infty}S_t=\infty$. 
                Then the first hitting time process $T$ of $S$ is a finite $\cF$-time-change such that $\lim_{t\to\infty}T_t=\infty$. 
                Moreover, if $S$ is strictly increasing, then $T$ has continuous paths. 
                \label{SETTING1}
          \item Let $T$ be a finite $\cF$-time-change such that $\lim_{t\to\infty}T_t=\infty$. 
                Then the first hitting time process $S$ of $T$ is a nondecreasing $\cF$-semimartingale such that 
                $\lim_{t\to\infty}S_t=\infty$.  Moreover, if $T$ has continuous paths, then $S$ is strictly increasing.  
                \label{SETTING2}
     \end{enumerate}
\end{Lem}

\begin{proof}   
     (1) The assumption $\lim_{t\to\infty}S_t=\infty$ implies that each random variable $T_t$ is finite.  
     In addition, since each $S_t$ is a real-valued random variable, it follows that $\lim_{t\to\infty}T_t=\infty$. 
     Fix $t\geq 0$.  Since $S$ is $\cF$-adapted, $\{ T_t < s\}=\{ S_{s-} > t \}\in{\cal F}_{s-}\subset{\cal F}_s$ for any $s>0$,  
     and obviously $\{T_t<0\}=\emptyset \in{\cal F}_0$.  
     Hence, $T_t$ is an $\cF$-optional time.  It follows from the right-continuity of $\cF$ that $T_t$ is an $\cF$-stopping time (see \cite[Prop.\ 1.2.3]{Karatzas-S}).    
     Thus, $T$ is a finite $\cF$-time-change.  
     Moreover, if $S$ is strictly increasing, then $T$ obviously has continuous paths. 
     
     (2) The assumption $\lim_{t\to\infty}T_t=\infty$ implies that each random variable $S_t$ is finite.  
     In addition, since each $T_t$ is a real-valued random variable, it follows that $\lim_{t\to\infty}S_t=\infty$. 
     Fix $s\geq 0$.  For any $t>0$, since $T_{t-}$ is also an $\cF$-stopping time, $\{S_s\geq t\}=\{T_{t-}\leq  s\}\in {\cal F}_s$.  
     Also, $\{S_s\geq  0\}=\Omega\in {\cal F}_s$.  Hence, $S_s$ is $\mathcal{F}_s$-measurable.  
     Therefore, $S$ is $\cF$-adapted.  Since $S$ is also c\`adl\`ag and nondecreasing, it is an $\cF$-semimartingale.    
     Moreover, if $T$ has continuous paths, then it is clear that $S$ is strictly increasing.      
\end{proof}

\begin{Rems}\label{Rems SETTING} \samepage \par
\begin{em} \mbox{ }
     \begin{enumerate}[(a)]
          \item
          Lemma \ref{Lem SETTING} establishes that a nondecreasing $\cF$-semimartingale $S$ and a finite $\cF$-time-change $T$ 
          are `dual' in the sense that either process with the specified condition induces the other. \label{SETTING3}
          \item 
          Part (\ref{SETTING1}) of Lemma \ref{Lem SETTING} assumes that $\lim_{t\to\infty}S_t=\infty$, which ensures that $T$ does not blow up in finite time.  
          We may lift this condition by restricting attention to $T_t$ with $t\in[0,t_\ast)$ 
          where $t_\ast=\sup_{0\leq s<\infty}S_s$, the explosion time of $T$.  
          The same argument applies to the assumption on $T$ in Part (\ref{SETTING2}). \label{SETTING4}
     \end{enumerate}
\end{em}
\end{Rems}

\begin{Notation}\label{Notations}\par
\begin{em}
     In light of Remark \ref{Rems SETTING} (\ref{SETTING3}),
     for a pair of a nondecreasing $\cF$-semimartingale $S$ and a finite $\cF$-time-change $T$, 
     \StoT\ and \TtoS\ are used to indicate respectively that $S$ induces $T$ and that $T$ induces $S$ 
     as described in Lemma \ref{Lem SETTING}.
     If $S$ is strictly increasing and $T$ has continuous paths, 
     then the double brackets \SttoT\ and \TttoS\ are employed instead.  
     Hence, the double bracket notation assumes stronger conditions than the single bracket notation.
     Hereafter, the notation $D$ and $E$ will be used to denote a pair of a strictly increasing semimartingale  
     and a continuous time-change. 
     This notation is chosen to be compatible with the continuous time-change $E$, which is induced by 
     a strictly increasing, stable subordinator $D$ of index between $0$ and $1$, in the papers 
     of Meerschaert and Scheffler~\cite{M-S_1,M-S_2} on continuous time random walks. 
\end{em}
\end{Notation}

\section{Stochastic Calculus for Stochastic Integrals Driven by a Time-changed Semimartingale} \label{sec results}

\noindent
     This section establishes a stochastic calculus for integrals driven by a time-changed semimartingale.
     The central problem is to understand such integrals by rephrasing them in terms of integrals driven by the original semimartingale. 
     Solving this problem is almost equivalent to providing a way to recognize SDEs 
     driven by a time-changed semimartingale, which aids the analysis of problems that appear in applications.  

     The following theorem, at first glance, may seem quite simple, but its impact on the formulation of our stochastic calculus is profound.  
     Recall that all processes, unless specified otherwise, are assumed to take values in $\bbR$ and start at $0$ throughout the paper.  
 
\begin{Thm}\label{Thm COV}\par \hspace{-2mm}
     \textbf{\emph{(2nd Change-of-Variable Formula)}} \
     Let $Z$ be an $\cF$-semimartingale.  
     Let $S$ and $T$ be a pair satisfying \StoT\ or \TtoS.
     Suppose $Z$ is in synchronization with $T$.
     If $K\in {L(Z\circ T,{\cal G}_t)}$, then $\bigl(K_{S(t-)}\bigr)\in L(Z, {{\cal G}_{S_t}})$ where ${\cal G}_t\DEF {\cal F}_{T_t}$. 
     Moreover, with probability one, for all $t\geq 0$,
     \begin{align}
          \int_0^t K_s dZ_{T_s} = \int_0^{T_t}K_{S(s-)}dZ_s . \label{COV1}
     \end{align}
\end{Thm}

\begin{proof}   
     By Lemma \ref{Lem JACOD1}, both $T$ and $Y\DEF Z\circ T$ are $\cG$-semimartingales.  
     Since $T$ is a nondecreasing $\cG$-semimartingale such that $\lim_{t\to\infty}T_t=\infty$ and $T_0=0$, 
     it follows from Part (1) of Lemma \ref{Lem SETTING} along with Remark \ref{Rems SETTING} (\ref{SETTING4}) 
     that $S$ is a finite $\cG$-time-change.  
     On any half open interval $[S_{s-},S_s)$, $T$ is obviously constant by construction and hence so is $Y$.  Moreover, 
     since $Z\synch T$, 
     \begin{align*}
          (Z\circ T)_{S(s)}=Z_{T(S(s))}=Z_{T(S(s)-)}=Z_{T(S(s-))}=(Z\circ T)_{S(s-)}.
     \end{align*}   
     Hence, $Y_{S_s}=Y_{S(s-)}$.  Thus, $Y$ is constant on any closed interval $[S_{s-},S_s]$.  Therefore, $Y\synch S$.    

     Now, let $K\in L(Y,{{\cal G}_t})$.  Then it follows from Lemma \ref{Lem JACOD2} that 
     $\bigl(K_{S(t-)}\bigr)\in L(Y\circ S, {{\cal G}_{S_t}})$.  
     By the 1st change-of-variable formula \eqref{JACOD2} and the assumption $Z\synch T$,  with probability one 
     \begin{align*}
          \int_0^{S_t} K_s dY_s=\int_0^t K_{S(s-)} dY_{S_s}=\int_0^t K_{S(s-)} dZ_{T(S(s))} = \int_0^t  K_{S(s-)} dZ_s
     \end{align*}          
     for all $t\geq 0$.  Hence, with probability one,
     \begin{align}
          \int_0^{S_{T_t}} K_s dY_s=\int_0^{T_t} K_{S(s-)} dZ_s \label{COV2}
     \end{align}
     for all $t\geq 0$.  Since $Y\synch S$, Lemma \ref{Lem SYNCH} yields $K\bullet Y\synch S$. 
     Any $t$ is contained in the interval $[S_{T(t)-},S_{T_t}]$, so $(K\bullet Y)_{S_{T_t}}=(K\bullet Y)_t$. 
     Thus, \eqref{COV2} establishes \eqref{COV1}.      
\end{proof}

\begin{Rems}\label{Rems COV}\par
\begin{em}\mbox{ }
     \begin{enumerate}[(a)]
          \item Theorem \ref{Thm COV} guarantees that \textit{any} stochastic integral driven by a time-changed semimartingale 
                is a time-changed stochastic integral
                driven by the original semimartingale, 
                as long as the semimartingale is in synchronization with the time-change. 
          \item If a pair $D$ and $E$ satisfies \DttoE\ or \EttoD, then \textit{any} process $Z$ is automatically 
                in synchronization with $E$ due to the continuity of $E$. 
                Therefore, under either of these stronger conditions,  Theorem \ref{Thm COV}  is valid for an arbitrary $\cF$-semimartingale $Z$.
                \label{COV3}
     \end{enumerate}
\end{em}
\end{Rems}

     In light of Remark \ref{Rems COV} (\ref{COV3}), when \DttoE\ or \EttoD, 
     the It\^o formula for stochastic integrals driven by a time-changed semimartingale 
     can be reformulated in a nice way via the 2nd change-of-variable formula \eqref{COV1} obtained in Theorem \ref{Thm COV}.  
     The proof of Theorem \ref{Thm ITO} is provided in full detail since it demonstrates important computational techniques 
     on quadratic variations which are frequently employed
     in Section \ref{sec LSDE}.   

\begin{Thm}\label{Thm ITO} \par \hspace{-2mm}
     \textbf{\emph{(Time-changed It\^o Formula)}} \
     Let $Z$ be an $\cF$-semimartingale.  Let $D$ and $E$ be a pair satisfying \DttoE\ or \EttoD.
     Define a filtration $\cG$ by ${\cal G}_t\DEF {\cal F}_{E_t}$. 
     Let $X$ be a process defined by 
     \begin{align}
          X_t:
          =(A\bullet m)_t+(F\bullet E)_t+\bigl(G\bullet (Z\circ E)\bigr)_t \label{ITO1}
          =\int_0^t A_s ds+\int_0^t F_s dE_s+ \int_0^t G_s dZ_{E_s} 
     \end{align} 
     where $A\in L(m,{{\cal G}_t})$, $F\in L(E,{{\cal G}_t})$, $G\in {L(Z\circ E,{\cal G}_t)}$, and $m$ is the identity map on $\bbR$
     corresponding to Lebesgue measure.
     If $f: \bbR\longrightarrow \bbR$ is a $C^2$ function,  
     then $f(X)$ is a $\cG$-semimartingale, and with probability one, for all $t\geq 0$, 
     \begin{align}  
          f(X_t)-f(0)
          =&\int_0^t f'(X_{s-}) A_s ds+\int_0^{E_t} f'\bigl(X_{D(s-)-}\bigr) F_{D(s-)} ds \label{ITO2} \vspace{1mm}\\
          +&\int_0^{E_t} f'\bigl(X_{D(s-)-}\bigr) {G}_{D(s-)} dZ_s
              + \dfrac 12 \int_0^{E_t} f''\bigl(X_{D(s-)-}\bigr) \bigl\{{G}_{D(s-)}\bigr\}^2 d[Z,Z]_s^c \notag \vspace{1mm}\\
          +&\sum_{0<s\leq t}^{\mbox{}}\bigl\{ f(X_s)-f(X_{s-})-f'(X_{s-})\Delta X_s \bigr\}. \notag
     \end{align}
     In particular, if $Z$ is a standard Brownian motion $B$, then with probability one, for all $t\geq 0$, 
     \begin{align}  
          f(X_t)-f(0)
          =\ &\int_0^t f'(X_s) A_s ds+\int_0^{E_t} f'\bigl(X_{D(s-)}\bigr) F_{D(s-)} ds \label{ITO3} \vspace{1mm}\\
          &+\int_0^{E_t} f'\bigl(X_{D(s-)}\bigr) {G}_{D(s-)} dB_s
                 + \dfrac 12 \int_0^{E_t} f''\bigl(X_{D(s-)}\bigr) \bigl\{{G}_{D(s-)}\bigr\}^2 ds. \notag 
     \end{align}
\end{Thm}

\begin{proof}   
     Since the process $X$ in \eqref{ITO1} is defined to be a sum of stochastic integrals driven by $\cG$-semimartingales, 
     $X$ itself is also a $\cG$-semimartingale by Property \ref{Properties} (\ref{*1}).  
     The It\^o formula \eqref{ItoFormula1} with $d=1$ yields, for all $t\ge 0$,
     \begin{align}  
       f(X_t)-f(0)&=\int_{0}^t f'(X_{s-}) dX_s+ \dfrac 12 \int_{0}^t f''(X_{s-}) d[X,X]_s^c \label{ITO11} \\
                     & \ \ \ +\sum_{0<s\leq t}\bigl\{ f(X_s)-f(X_{s-})-f'(X_{s-})\Delta X_s \bigr\}. \notag
     \end{align}
     Using Properties \ref{Properties} (\ref{*3}), (\ref{*4}) and the 2nd change-of-variable formula \eqref{COV1}, 
     \begin{align}
          \int_{0}^t f'(X_{s-})  dX_s \label{ITO13}
          &=\int_{0}^t f'(X_{s-})  A_s ds +\int_{0}^t f'(X_{s-})  F_s dE_s + \int_{0}^t f'(X_{s-})  G_s dZ_{E_s} \\
          =\int_{0}^t f'(X_{s-})  &A_s ds +\int_{0}^{E_t} \hspace{-1mm}f'\bigl(X_{D(s-)-}\bigr) F_{D(s-)} ds 
                          + \int_{0}^{E_t}\hspace{-1mm} f'\bigl(X_{D(s-)-}\bigr)  G_{D(s-)} dZ_s. \notag
     \end{align}

     For the second integral on the right hand side of \eqref{ITO11}, first let $Y\DEF Z\circ E$.   We claim that  
     \begin{align}
          [X,X]_t^c=\int_0^t G_s^2 d[Y,Y]_s^c. \label{ITO14}
     \end{align}
     To prove this, first note that $m$ and $E$ are both continuous processes of finite variation on compact sets.  
     By \cite[II.\ Thm.\ 26]{Protter},  
     \begin{align*}
          [m,Y]_t=\sum_{0<s\leq t}\Delta[m,Y]_s=\sum_{0<s\leq t}(\Delta m_s)\cdot (\Delta X_s)=0 
     \end{align*}
     for all $t\geq 0$.  Hence, $[m,Y]=0$.  Similarly, $[m,m]=[m,E]=[E,E]=[E,Y]=0$.  
     Therefore, the bilinearity of $[\cdot,\cdot]$ and Property \ref{Properties} (\ref{*5}) imply
     \begin{align}
          [X,X]=[A\bullet m + F\bullet E + G\bullet Y, \hspace{1pt}A\bullet m + F\bullet E + G\bullet Y]=G^2\bullet [Y,Y].\label{ITO15}
     \end{align}
     Now, let $J_t\DEF \sum_{0<s\leq t}\Delta [Y,Y]_s$ so that $[Y,Y]_t^c=[Y,Y]_t-J_t$. 
     Then the pure jump process, $J$, shares with  $[Y,Y]$  the same jump times and sizes.  Therefore, 
     \begin{align*}
          \sum_{0<s\leq t} G_s^2 \Delta [Y,Y]_s
          =\sum_{0<s\leq t} G_s^2 \Delta J_s
          =\int_0^t G_s^2 dJ_s.
     \end{align*}
     Hence, it follows from \eqref{ITO15} together with Properties \ref{Properties} (\ref{*2}), (\ref{*3}) that 
     \begin{align*}
          [X,X]_t^c
          &=[X,X]_t - \sum_{0<s\leq t} \Delta [X,X]_s =(G^2\bullet [Y,Y])_t - \sum_{0<s\leq t} G_s^2 \Delta[Y,Y]_s  \\
          &=\int_0^t G_s^2 d[Y,Y]_s -\int_0^t G_s^2 dJ_s =\int_0^t G_s^2 d[Y,Y]_s^c,
     \end{align*}
     thereby establishing \eqref{ITO14}.   
     
     Since $Z\synch E$, repeated use of Lemma \ref{Lem JACOD3} yields 
     \begin{align}
          [Y,Y]^c =[Y^c,Y^c]=[Z^c\circ E,Z^c\circ E]=[Z^c,Z^c]\circ E =[Z,Z]^c\circ E. \label{ITO16}
     \end{align} 
     Together \eqref{ITO14} and \eqref{ITO16} yield                        
          $[X,X]_t^c=\int_0^t G_s^2 d[Z,Z]_{{}_{E_s}}^c$. 
     Therefore, it follows from Property \ref{Properties} (\ref{*4}) and the 2nd change-of-variable formula \eqref{COV1} that      
     \begin{align}
          \int_{0}^t f''(X_{s-}) d[X,X]_s^c &= \int_0^t f''(X_{s-}) G_s^2 d[Z,Z]_{{}_{E_s}}^c   \label{ITO17}\\
           &=\int_0^{E_t} f''\bigl(X_{D(s-)-}\bigr) \{G_{D(s-)}\}^2 d[Z,Z]_s^c.     \notag
     \end{align}
     Equality \eqref{ITO2} follows by plugging \eqref{ITO13} and \eqref{ITO17} into Formula \eqref{ITO11}. 
     
     If $Z=B$ is a standard Brownian motion, then the continuity of $m$, $E$ and $B\circ E$ together with 
     Property \ref{Properties} (\ref{*2})
     imply $X$ is also continuous.   Since $[B,B]_t^c=[B,B]_t=t$, statement \eqref{ITO3} follows immediately.     
\end{proof}

     A similar proof yields the multidimensional version of Theorem \ref{Thm ITO}.   
     For a multidimensional process $W$, its $i$-th component is denoted $W^i$. 

\begin{Cor}\label{Cor ITO}\par
     Let $Z$ be an $n$-dimensional $\cF$-semimartingale starting at $0$.  
     Let $D$ and $E$ be a pair satisfying \DttoE\ or \EttoD.
     Define a filtration $\cG$ by ${\cal G}_t\DEF {\cal F}_{E_t}$. 
     Let $X$ be a $d$-dimensional process defined by 
     \begin{align*}
          X_t\DEF \int_0^t A_s ds+\int_0^t F_s dE_s+ \sum_{k=1}^n \int_0^t G_s^k dZ_{E_s}^k 
     \end{align*} 
     where $A$, $F$ and $G^k=(G^{k,1},\ldots,G^{k,d})$ $(k=1,\ldots,n)$ are d-dimensional processes 
     for which all the above integrals are defined.  
     If $f: \bbR^d \longrightarrow \bbR$ is a $C^2$ function,  
     then $f(X)$ is a $\cG$-semimartingale, and with probability one, for all $t\geq 0$, 
     \begin{align}  
        f(X_t)-f(0)\label{ITO18}
        &=\sum_{i=1}^d\int_0^t \frac{\partial f}{\partial x^i}(X_{s-}) A_s^i ds
              +\sum_{i=1}^d\int_0^{E_t} \frac{\partial f}{\partial x^i}\bigl( X_{D(s-)-}\bigr) F_{D(s-)}^i ds \vspace{1mm}\\
        & \ \ \ +\sum_{i=1}^d\sum_{k=1}^n\int_0^{E_t} \frac{\partial f}{\partial x^i}\bigl(X_{D(s-)-}\bigr) 
                G_{D(s-)}^{k,i} dZ_s^k \notag \vspace{1mm}\\
        & \ \ \ + \dfrac 12 \sum_{i,j=1}^d \sum_{k,\ell=1}^n \int_0^{E_t} \frac{\partial^2 f}{\partial x^i \partial x^j}\bigl(X_{D(s-)-}\bigr) 
                          {G}_{D(s-)}^{k,i} {G}_{D(s-)}^{\ell,j} d[Z^k,Z^\ell]_s^c \notag  \vspace{1mm}\\
        & \ \ \ +\sum_{0<s\leq t}\Bigl\{ f(X_s)-f(X_{s-})-\sum_{i=1}^d \dfrac{\partial f}{\partial x^i}(X_{s-})\Delta X_s^i \Bigr\}. \notag
     \end{align}
\end{Cor}

\begin{Rems}\label{Rems ITO}\par
\begin{em}\mbox{ }
     \begin{enumerate}[(a)]
          \item The first integral in Formula \eqref{ITO2} can also be expressed as a time-changed stochastic integral.  
          By the 2nd change-of-variable formula \eqref{COV1}, 
          \begin{align}
               \int_0^tf'(X_{s-})A_s ds \label{ITO21}
               &= \int_0^tf'(X_{s-})A_s dD_{E_s} +\sum_{0<s\leq t}f'(X_{s-})A_s\Delta(D\circ E)_s \\
               &=\int_0^{E_t} f'\bigl(X_{D(s-)-}\bigr) A_{D(s-)} dD_s +\sum_{0<s\leq t}f'(X_{s-})A_s\Delta(D\circ E)_s\notag
          \end{align}
          as long as all integrals are defined.  The additional term arises due to the discontinuities of $D$. 
          \item The stronger condition \DttoE\ or \EttoD, rather than \DtoE\ or \EtoD, is essential 
          in establishing the nice representations \eqref{ITO2} and \eqref{ITO18}.  
          For example, if $E$ has jumps, then the stochastic integral
          $\int_{0}^t f'(X_{s-})  F_s dE_s$ 
          in \eqref{ITO13} may not be rephrased 
          as a time-changed integral driven by $ds$ since the identity map $m(s)=s$ is no longer in synchronization with $E$.  
          Moreover, the equalities $[E,E]=0$
          and $[E,Y]=0$ both may fail, which implies more terms need to be included in \eqref{ITO14}.      
          \item In real situations, the distributions of 
          $Z$, $D$ and $E$ are known 
          through statistical data, and scientists 
          will seek to reveal the behavior of a process $X$ described via an SDE of the form 
          \begin{align}
               dX_t=\rho(t,E_t,X_t)dt+\mu(t,E_t,X_t)dE_t+\sigma(t,E_t,X_t)dZ_{E_t}. \label{ITO22}
          \end{align}
          Formula \eqref{ITO2} encourages handling the solution to Equation \eqref{ITO22} 
          via conditioning.  In particular, when $Z$ is continuous and $A\equiv 0$, the right hand side of Formula \eqref{ITO2},
          conditioned on $E_t$, 
          can be regarded as usual stochastic integrals driven simply by Lebesgue measure, $Z$ and $[Z,Z]$.  \label{ITO23}
     \end{enumerate}
\end{em}
\end{Rems}
     
     The following example provides a sense of the kinds of results that can be obtained using Theorem \ref{Thm ITO} 
     together with conditioning.  
     
\begin{Ex}\label{Ex FP}\par
\begin{em}
     Let $D$ be an $\cF$-stable subordinator of index $\beta\in(0,1)$ which is 
     independent of a standard $\cF$-Brownian motion $Z=B$.  
     The process $D$ is strictly increasing. 
     Let $E$ be the associated continuous time-change so that \DttoE.
     Then under a certain condition, 
     the transition probability density 
     $p^X(t,y)\equiv p^X(t,y|0,x)$
     of a solution $X$ to the SDE
     \begin{align}
          dX_t= \mu(X_t)dE_t+\sigma(X_t)dB_{E_t} \ \ \ \textrm{with} \ X_0=x \label{FP1}
     \end{align}
     satisfies the following time-fractional PDE in the weak sense:
     \begin{align}
          {\bf D}_\ast^\beta p^X(t,y) = -\dfrac\partial {\partial y}\bigl\{ \mu(y) p^X(t,y) \bigr\}
               + \dfrac 12 \dfrac{\partial^2} {\partial y^2} \bigl\{\sigma^2(y) p^X(t,y) \bigr\}, \label{FP2}
     \end{align}
     with initial condition $p^X(0,y)=\delta_x(y)$.  
     Here, ${\bf D}_\ast^\beta$ is the Caputo fractional derivative
     of order $\beta$ with respect to the time variable $t$ (see \cite{G-M}), 
     and $\delta_x$ is the Dirac delta function with mass at $x$.  
     For the proof, see Hahn, Kobayashi and Umarov~\cite[Thm.\ 4.1]{HKU}. 
     Furthermore, that paper provides a more general perspective on this matter in the framework of 
     time-changed L\'evy processes and their associated pseudo-differential equations, with the time-change being the first hitting time process of a mixture of independent stable subordinators. 
     Moreover, the above result is derived there without the use of Theorem \ref{Thm ITO}, 
     but based on Theorem \ref{Thm SDE-duality} of the present paper.  
     The advantage of the approach which employs 
     the time-changed It\^o formula \eqref{ITO3}
     is that it reveals the connection 
     between the stochastic calculus for a time-changed Brownian motion and 
     its associated time-fractional PDE \eqref{FP2}. 
     A further remark on Equations \eqref{FP1} and \eqref{FP2} is provided in this paper in Example \ref{Ex LSDE-8}. \qed 
\end{em}
\end{Ex}

\begin{Rem}
\begin{em}
   With regards to Example \ref{Ex FP}, it is possible to discuss SDEs and associated time-fractional PDEs \textit{with smooth boundary conditions}.  
Time-fractional PDEs with Dirichlet boundary conditions are treated in \cite{MNV}, but without specifying the connection to SDEs of the form \eqref{FP1}.   
\end{em}
\end{Rem}

\section{SDEs Including Terms Driven by a Time-changed Semimartingale} \label{sec LSDE}

     A classical It\^o SDE is of the form 
     \begin{align*}
          dY_t=b(t,Y_t)dt+\tau(t,Y_t)dB_t
     \end{align*} 
     where $B$ is a standard Brownian motion.   
     As stated in Remark \ref{Rems ITO} (\ref{ITO23}),
     the 2nd change-of-variable formula \eqref{COV1}
     is a useful tool in handling a larger class of SDEs of the form
     \begin{align*}
          dX_t=\rho(t,E_t,X_t)dt+\mu(t,E_t,X_t)dE_t+\sigma(t,E_t,X_t)dB_{E_t},
     \end{align*}
     where $E$ is a continuous time-change. 
     Note that the sample path $t\mapsto E_t$ is not necessarily absolutely continuous with respect to Lebesgue measure; 
     hence, the $dE_t$ term appearing above in general cannot be rewritten in terms of $dt$. 
     For example, if $E$ is the first hitting time process of 
     a stable subordinator $D$ of index between $0$ and $1$,
     then the sample path $t\mapsto E_t$ is flat almost everywhere.  Therefore, if $E_t$ had a representation $E_t=\int_0^t g(s)ds$
     for some integrable function $g$, then it would follow that $E_t=0$ for all $t\geq 0$, contradicting the fact that 
     $\lim_{t\to\infty}E_t=\infty$. 
     More generally, if $E$ is the first hitting time process of a strictly increasing L\'evy process 
     with infinite jumps and no drift, then $E$ is not absolutely continuous with respect to Lebesgue measure.  
     For definition and properties of L\'evy processes, consult \cite{Applebaum} or \cite{Sato}.
     
     The new feature of this larger class of SDEs is the coexistence of a usual drift term 
     along with a term representing a 
     factor ascribed to the time-change.  
     The aim of this section is to provide ways of recognizing this new larger class of SDEs by analyzing
     their solutions and making comparisons between the two classes of SDEs. 
     For a general treatment of classical It\^o SDEs, see \cite{Karatzas-S} or \cite{Oksendal}. 
     Regarding methods for obtaining explicit forms of solutions to classical It\^o SDEs, consult \cite[Chap.\ 4]{Gard}.   
     Many basic models are introduced in \cite{Steele} with an abundance of interpretations and insights.     

     Let $Z$ be an $\cF$-semimartingale and let $E$ be a continuous $\cF$-time-change.  
     The general form of SDEs discussed here is 
     \begin{multline}
          dX_t=\rho(t,E_t,X_{t-})dt+\mu(t,E_t,X_{t-})dE_t+\sigma(t,E_t,X_{t-})dZ_{E_t}
               \ \ \textrm{with} \ X_0=x_0, \label{SDE1}
     \end{multline}
     which is understood in the following integral form: 
     \begin{align}
          X_t=x_0+\int_0^t \rho(s,E_s ,X_{s-})ds+\int_0^t \mu(s,E_s ,X_{s-})dE_s \label{SDE1.5}
                 + \int_0^t \sigma(s,E_s ,X_{s-})dZ_{E_s},
     \end{align}
     where $x_0$ is a real constant, and $\rho$, $\mu$, $\sigma$ are real-valued functions, defined on 
     $\bbR_+\ttimes \bbR_+ \ttimes \bbR$, which 
     satisfy the following \textit{Lipschitz condition}: 
    there exists a positive constant $L$ such that 
     \begin{align}
          |\rho(t,u,x)-\rho(t,u,y)|+ |\mu(t,u,x)-\mu(t,u,y)| \label{SDE2}
            + |\sigma(t,u,x)-\sigma(t,u,y)|\leq L|x-y|
     \end{align}
     for all $t, u\in\bbR_+$ and $x, y\in\bbR$.  
     For technical reasons, we require
     {\it assumption} 
     \begin{align}
          \label{SDE2.5} X\in \bbD\cG \Longrightarrow   
                \bigl(\rho(t,E_t,X_{t-})\bigr),\, \bigl(\mu(t,E_t,X_{t-})\bigr),\, \bigl(\sigma(t,E_t,X_{t-})\bigr) \in \bbL\cG, 
     \end{align}
     where $\mathcal{G}_t\DEF \mathcal{F}_{E_t}$.   
     One example of such functions is a `linear' map 
     $\rho(t,u,x)=\rho_1(t,u)+\rho_2(t,u)\cdot x$, 
     where $\rho_1$, $\rho_2$ are bounded continuous functions on $\bbR_+\ttimes\bbR_+$. 
     
\begin{Lem}\label{Lem SDE-existence}\par \hspace{-2mm}
     \textbf{\emph{(Existence and Uniqueness of Solution)}} \
     Let $Z$ be an $\cF$-semimartingale.  Let $D$ and $E$ be a pair satisfying \DttoE\ or \EttoD. 
     Suppose $\rho$, $\mu$, $\sigma$ are real-valued functions defined 
     on $\bbR_+\ttimes \bbR_+ \ttimes \bbR$ 
     satisfying Lipschitz condition \eqref{SDE2} and assumption \eqref{SDE2.5}. 
     Then there exists a unique $\cG$-semimartingale $X$ for which \eqref{SDE1} holds, where $\mathcal{G}_t\DEF \mathcal{F}_{E_t}$. 
\end{Lem}

\begin{proof}   
     The identity map $m$ corresponding to Lebesgue measure can be regarded as a $\cG$-semimartingale.   
     Moreover, $E$ and $Z\circ E$ are also $\cG$-semimartingales due to Lemma \ref{Lem JACOD1}. 
     The existence and uniqueness of a strong solution $X$ to SDE \eqref{SDE1} is guaranteed by 
     conditions \eqref{SDE2} and \eqref{SDE2.5}, upon reformulating Theorem 7 of \cite[Chap.\ V]{Protter}  
     with operators $F_j:\bbD\cG\longrightarrow \bbL\cG$ $(j=1,2,3)$ defined by
     \begin{align*}
          F_1(X)_t=\rho(t,E_t,X_{t-}), \ \ \  F_2(X)_t=\mu(t,E_t,X_{t-}), \ \ \  F_3(X)_t=\sigma(t,E_t,X_{t-}).
     \end{align*} 
     Furthermore, it follows from Property \ref{Properties} (\ref{*1}) and the integral expression \eqref{SDE1.5} 
     that $X$ is a $\cG$-semimartingale.  
\end{proof}

     Now that the existence and uniqueness of a solution to an SDE of the form \eqref{SDE1} is established, 
     the following two SDEs both make sense:  
     \begin{align}
          dX_t&=\mu(E_t,X_{t-})dE_t+\sigma(E_t,X_{t-})dZ_{E_t} \ \ \ \textrm{with} \ X_0= x_0; \label{SDE3}\\
          dY_t&=\mu(t,Y_{t-})dt+\sigma(t,Y_{t-})dZ_t \ \ \ \textrm{with} \ Y_0= x_0.  \label{SDE4}
     \end{align}
     Together the change-of-variable formulas \eqref{JACOD2} and \eqref{COV1}  
     yield Theorem \ref{Thm SDE-duality}, which in turn
     reveals a close connection between the classical It\^o-type SDE \eqref{SDE4} and our new class of SDEs in \eqref{SDE3}.  
   
\begin{Thm}\label{Thm SDE-duality}\hspace{-2mm}
     \textbf{\emph{(Duality of SDEs)}} \
     Let $Z$ be an $\cF$-semimartingale.   Let $D$ and $E$ be a pair satisfying \DttoE\ or \EttoD. \vspace{1mm} \\
          (1) \hspace{1pt} If a process $Y$ satisfies SDE \eqref{SDE4}, then $X\DEF Y\circ E$ satisfies SDE \eqref{SDE3}.\\
          (2) \hspace{1pt} If a process $X$ satisfies SDE \eqref{SDE3}, then $Y\DEF X\circ D$ satisfies SDE \eqref{SDE4}. 
\end{Thm}

\begin{proof}   
     (1) Suppose $Y$ satisfies SDE \eqref{SDE4}, and let $X\DEF Y\circ E$.
     Since any process is in synchronization with the continuous $\cF$-time-change $E$, 
     the 1st change-of-variable formula \eqref{JACOD2} yields
     \begin{align}
          X_t&=x_0+\int_0^{E_t}\mu(s,Y_{s-})ds+\int_0^{E_t}\sigma(s,Y_{s-})dZ_s \label{SDE5}\\
             &=x_0+\int_0^t \mu\bigl(E_s,Y_{E(s)-}\bigr)dE_{s}+\int_0^t\sigma\bigl(E_s,Y_{E(s)-}\bigr)dZ_{E_s}. \notag
     \end{align}
     In general, the equality $Y_{E(s)-}=(Y\circ E)_{s-}$ may fail.  
     The failure can occur only when $E$ is constant on some interval $[s-\varepsilon,s]$ with $\varepsilon>0$.
     However, the integrators $E$ and $Z\circ E$ on the right hand side of \eqref{SDE5} are constant on this interval; 
     hence, the difference between the two values $Y_{E(s)-}$ and $(Y\circ E)_{s-}$ does not affect the value of the integrals.  
     Thus, \eqref{SDE5} can be reexpressed as
     \begin{equation}
          X_t=x_0+\int_0^t \mu\bigl(E_s,(Y\circ E)_{s-}\bigr)dE_s+\int_0^t\sigma\bigl(E_s,(Y\circ E)_{s-}\bigr)dZ_{E_s}, 
     \end{equation}
     thereby yielding SDE \eqref{SDE3}.  

     (2) Next, suppose $X$ satisfying SDE \eqref{SDE3} is given.  Since $D$ is strictly increasing, 
     $X_{D(s-)-}=(X\circ D)_{s-}$ for any $s>0$.  
     Again, since any process is in synchronization with the continuous $\cF$-time-change $E$, 
     the 2nd change-of-variable formula \eqref{COV1} applied to the integral form of SDE \eqref{SDE3} yields 
     \begin{align}
          X_t&=x_0+\int_0^{E_t} \mu\bigl(E_{D(s-)},X_{D(s-)-}\bigr)ds+\int_0^{E_t}\sigma\bigl(E_{D(s-)},X_{D(s-)-}\bigr)dZ_s \label{SDE6}\\
              &=x_0+\int_0^{E_t} \mu\bigl(s,(X\circ D)_{s-}\bigr)ds+\int_0^{E_t}\sigma\bigl(s,(X\circ D)_{s-}\bigr)dZ_s. \notag
     \end{align} 
     Let $Y\DEF X\circ D$, then \eqref{SDE6} immediately yields SDE \eqref{SDE4}, which completes the proof.      
\end{proof}

\begin{Rem}\label{Rem SDE-duality}\par
\begin{em}
     One may wonder whether the SDE 
          \begin{align*}
               dX_t=\rho(E_t,X_{t-})dt+\mu(E_t,X_{t-})dE_t+\sigma(E_t,X_{t-})dZ_{E_t} 
                \ \ \ \textrm{with} \ X_0= x_0 
          \end{align*}
     can be reduced in the same manner as Theorem \ref{Thm SDE-duality} (2). 
     This is a question of whether the new driving process $dt$ can be replaced by $dD_{E_t}$, 
     which is possible only in very special cases; e.g., 
     if $D$ is continuous or $\rho(E_t,X_{t-})$ vanishes on every nonempty open interval $(D_{u-},D_u)$. \label{SDE7}
\end{em}
\end{Rem}

     For the remainder of this section, consideration mainly focuses on \textit{linear} SDEs of the form 
     \begin{align}
          dX_t&=\bigl(\rho_1(t,E_t) +\rho_2(t,E_t)X_t\bigr)dt+\bigl(\mu_1(t,E_t)+\mu_2(t,E_t)X_t\bigr)dE_t   \label{SDE21} \\
               & \ \ \ \ +\bigl(\sigma_1(t,E_t)+\sigma_2(t,E_t)X_t\bigr)dB_{E_t}
                \ \ \ \textrm{with} \ X_0=x_0. \notag
     \end{align}
     Here $B$ is a standard $\cF$-Brownian motion, $E$ is a continuous $\cF$-time-change, and   
     $\rho_j$, $\mu_j$, $\sigma_j$ $(j=1,2)$ are real-valued functions on $\bbR_+\ttimes\bbR_+$ satisfying the following conditions:
     \begin{align}
        &|\rho_2(t,u)|+ |\mu_2(t,u)|+ |\sigma_2(t,u)|\leq L \ \ \ \textrm{for all} \ t, u\in\bbR_+, \tag{\ref{SDE2}'}\label{SDE2'}\\
        &\bigl(\rho_j(t,E_t)\bigr),\, \bigl(\mu_j(t,E_t)\bigr),\, \bigl(\sigma_j(t,E_t)\bigr) \in \bbL\cG \ \ \ \textrm{for} \ j=1,2, 
              \tag{\ref{SDE2.5}'}\label{SDE2.5'}
     \end{align}     
     where $L$ is a positive constant and $\mathcal{G}_t\DEF \mathcal{F}_{E_t}$.  
     Note that a strong solution $X$ to SDE \eqref{SDE21}
     always has continuous paths due to the continuity of the driving processes. 
     Conditions \eqref{SDE2'} and \eqref{SDE2.5'} respectively imply conditions \eqref{SDE2} and \eqref{SDE2.5};
     therefore, the uniqueness and existence of the strong solution $X$ is guaranteed by Lemma \ref{Lem SDE-existence}.
     
     As demonstrated in the proof of Theorem \ref{Thm ITO}, we have the handy calculus rules 
     \begin{equation}\begin{cases}
          & \hspace{-2mm}[m,m]=[m,E]=[m,B\circ E]=[E,E]=[E,B\circ E]=0, \label{SDE20}\\  
          & \hspace{-2mm}[B\circ E,B\circ E]=E, 
     \end{cases}\end{equation}
     where $m$ denotes the identity map corresponding to Lebesgue measure. 
     Remark \ref{Rem SDE-duality} implies that the simple substitution $Y_t\DEF X_{D_t}$ fails to
     reduce even the most basic type of SDE \eqref{SDE21} into a classical It\^o SDE
     due to the presence of the $dt$ term.   
     This observation suggests that we establish a general form of solution to \eqref{SDE21} 
     via a direct approach rather than via such a simple substitution. 
     It also calls into question the possibility of developing reduction schemes 
     for converting SDEs of the form \eqref{SDE21} into less complicated SDEs.   
     Propositions \ref{Prop LSDE-2} and \ref{Prop Reduction} together with Theorem \ref{Thm LSDE-3} largely settle this issue.  
     The linear SDE \eqref{SDE21} is said to be \textit{homogeneous} if $\rho_1=\mu_1=\sigma_1\equiv 0$.  

\begin{Prop}\label{Prop LSDE-2}\par \hspace{-2mm}
     \textbf{\emph{(Solution Form for Homogeneous Linear SDEs)}} \
     Let $B$ be a standard $\cF$-Brownian motion.  Let $D$ and $E$ be a pair satisfying \DttoE\ or \EttoD. 
     Then the unique strong solution to the homogeneous linear SDE with initial condition
     \begin{equation}
          dX_t=\rho_2(t,E_t) X_tdt+\mu_2(t,E_t)X_tdE_t+\sigma_2(t,E_t)X_tdB_{E_t} \ \ \ \textrm{with} \ X_0= x_0 \label{SDE22} 
     \end{equation}
     is explicitly written as 
     \begin{align}
          X_t=x_0 &\exp\Biggl\{ \int_0^t \rho_2(s,E_s)ds+\int_0^t \Bigl( \mu_2(s,E_s)-\dfrac 12 \sigma_2^2(s,E_s) \Bigr) dE_s  \label{SDE23}\\
              +&\int_0^t \sigma_2(s,E_s) dB_{E_s}\Biggr\}, \notag
           \end{align}
     or equivalently as
     \begin{align}
          X_t=x_0 &\exp\Biggl\{ \int_0^t \rho_2(s,E_s)ds+\int_0^{E_t} \Bigl( \mu_2(D_{s-},s)-\dfrac 12 \sigma_2^2(D_{s-},s)\Bigr) ds \label{SDE24}\\
                     +&\int_0^{E_t} \sigma_2(D_{s-},s) dB_s\Biggr\}.  \notag
     \end{align}
\end{Prop}

\begin{proof}
     \eqref{SDE24} follows from \eqref{SDE23} together with the 2nd change-of-variable formula \eqref{COV1}.
     Due to the uniqueness of the solution, it suffices to show that the process $X$ given in \eqref{SDE23} 
     satisfies SDE \eqref{SDE22}.  

     Let $X$ be the process in \eqref{SDE23} and write $X_t=x_0 \, e^{A_t}$.
     A calculation similar to \eqref{ITO15}, via \eqref{SDE20}, yields $[A,A]=\sigma_2^2(\cdot,E) \bullet E$.
     By the It\^o formula \eqref{ItoFormula1} with $f(a)=x_0 \, e^a$, 
     \begin{align}
          dX_t&= x_0\, e^{A_t}dA_t+\tfrac 12 x_0\, e^{A_t} d[A,A]_t \label{SDE25} \\
              &= X_t\bigl\{\rho_2(t,E_t)dt+\bigl(\mu_2(t,E_t)-\tfrac 12 \sigma_2^2(t,E_t) \bigr)dE_t+\sigma_2(t,E_t)dB_{E_t}\bigr\}\notag \\
               & \ \ \ \ +\tfrac 12 X_t \sigma_2^2(t,E_t)dE_t \notag\\
              &= \rho_2(t,E_t) X_tdt+\mu_2(t,E_t)X_tdE_t+\sigma_2(t,E_t)X_tdB_{E_t}.\notag
     \end{align}
     In addition, $X_0=x_0$.  Thus, $X$ satisfies \eqref{SDE22}, completing the proof.
\end{proof}

\begin{Thm}\label{Thm LSDE-3}\par \hspace{-2mm}
     \textbf{\emph{(General Solution Form for Linear SDEs)}} \
     Let $B$ be a standard $\cF$-Brownian motion and $D$ and $E$ be a pair satisfying \DttoE\ or \EttoD. 
     Then the unique strong solution to a general linear SDE \eqref{SDE21} is explicitly written as 
     \begin{align}\label{SDE31}
          X_t=\ &\Phi_t \Biggl[ x_0+\int_0^t \dfrac{\rho_1(s,E_s)}{\Phi_s} ds\\
                                           &+\int_0^t \dfrac{ \mu_1(s,E_s)-\sigma_2(s,E_s)\sigma_1(s,E_s)} {\Phi_s} dE_s
                                           +\int_0^t \dfrac{\sigma_1(s,E_s)} {\Phi_s} dB_{E_s} \Biggr], \notag
     \end{align}
     or equivalently as
     \begin{align}\label{SDE32}
          X_t=\ &\Phi_t \Biggl[ x_0+\int_0^t \dfrac{\rho_1(s,E_s)}{\Phi_s}  ds\\
                                           &+\int_0^{E_t} \dfrac{ \mu_1(D_{s-},s)-\sigma_2(D_{s-},s)\sigma_1(D_{s-},s)} {\Phi_{D(s-)}} ds
                                           +\int_0^{E_t} \dfrac{\sigma_1(D_{s-},s)} {\Phi_{D(s-)}}dB_s \Biggr], \notag
     \end{align}
     where 
     $\Phi$ is the unique strong solution \eqref{SDE23} to the homogeneous linear SDE \eqref{SDE22} with $x_0$ replaced by $1$. 
     $\Phi$ is called the \textit{fundamental solution} to the homogeneous SDE \eqref{SDE22}.  
\end{Thm}

\begin{proof}   
     Since $\Phi_0=1>0$, the explicit form \eqref{SDE23} of $\Phi$
     shows that $\Phi_t>0$ for all $t\geq 0$.  Hence, the right hand side of \eqref{SDE31} is meaningful.   
     As in the proof of Proposition \ref{Prop LSDE-2}, it is sufficient to check that the process $X$ in \eqref{SDE31} 
     satisfies SDE \eqref{SDE21}.  
     For notational convenience, we suppress the dependence of the coefficients on $E_t$ and simply write 
     $\rho_j(t)=\rho_j(t,E_t)$, $\mu_j(t)=\mu_j(t,E_t)$, $\sigma_j(t)=\sigma_j(t,E_t)$ for $j=1,2$.      
     
     Let $X$ be the process in \eqref{SDE31} and write $X_t=\Phi_tZ_t$.          
     Since $\Phi$ is the solution to SDE \eqref{SDE22}, the calculus rule \eqref{SDE20} yields 
     $[\Phi,Z]= \bigl( \sigma_2 \Phi \cdot (\sigma_1/\Phi ) \bigr) \bullet E =(\sigma_2 \sigma_1) \bullet E$.       
     Hence, using the product formula \eqref{ItoFormula2},      
     \begin{align}
          dX_t=&\ \Phi_t dZ_t+Z_t d\Phi_t +d[\Phi,Z]_t\label{SDE34} \\ 
              = &\ \rho_1(t)dt+\bigl(\mu_1(t)-\sigma_2(t)\sigma_1(t) \bigr)dE_t + \sigma_1(t) dB_{E_t}\notag\\ 
                &+Z_t \bigl( \rho_2(t)\Phi_t dt +\mu_2(t)\Phi_t dE_t+ \sigma_2(t)\Phi_t dB_{E_t}\bigr) +\sigma_2(t)\sigma_1(t)dE_t, \notag
     \end{align}
     the right hand side of which yields that of SDE \eqref{SDE21} upon replacing $\Phi_t Z_t$ by $X_t$. 
     Moreover, $X_0=x_0$, completing the proof. 
\end{proof}  

     A multidimensional version of Theorem \ref{Thm LSDE-3} can be obtained in a similar way
     by applying the It\^o formula componentwise.   

\begin{Cor}\label{Cor LSDE-3}\par 
     Let $B$ be an $n$-dimensional standard $\cF$-Brownian motion starting at $0$. 
     Let $\bigl(\rho_2(t,E_t)\bigr)$, $\bigl(\mu_2(t,E_t)\bigr)$, 
     $\bigl(\sigma_2^k(t,E_t)\bigr)$ $(k=1,\ldots,n)$ be $d\ttimes d$-matrix-valued processes, 
     $\bigl(\rho_1(t,E_t)\bigr)$, $\bigl(\mu_1(t,E_t)\bigr)$, 
     $\bigl(\sigma_1^k(t,E_t)\bigr)$ $(k=1,\ldots,n)$ be $d$-dimensional processes.  
     Let $x_0\in\bbR^d$.  Then the unique strong solution $X$ to the SDE
     \begin{align}
          dX_t= \ &\bigl(\rho_1(t,E_t)+\rho_2(t,E_t)X_t\bigr)dt+\bigl(\mu_1(t,E_t)+\mu_2(t,E_t)X_t\bigr)dE_t \label{SDE38} \\
                &+\sum_{k=1}^n \bigl(\sigma_1^k(t,E_t)+\sigma_2^k(t,E_t)X_t\bigr)dB_{E_t}^k 
                 \ \ \ \textrm{with} \ X_0=x_0, \notag
     \end{align}
     which is a $d$-dimensional process, is explicitly written as 
     \begin{align}
           X_t=\ &\Phi_t \Biggl[ x_0+\int_0^t \Phi_s^{-1} \rho_1(s,E_s) ds\label{SDE39}\\
                            &+\int_0^t \Phi_s^{-1} \Bigl(\mu_1(s,E_s)-\sum_{k=1}^n \sigma_2^k(s,E_s)\sigma_1^k(s,E_s)\Bigr) dE_s
                                  \notag \\
                            &+\int_0^t \Phi_s^{-1} \sum_{k=1}^n\sigma_1^k(s,E_s) dB_{E_s}^k \Biggr], \notag
     \end{align}
     or equivalently as 
     \begin{align}
          X_t=\ &\Phi_t \Biggl[ x_0+\int_0^t \Phi_s^{-1} \rho_1(s,E_s) ds\label{SDE40}\\
                             &+\int_0^{E_t} \Phi_{D(s-)}^{-1} \Bigl(\mu_1(D_{s-},s)-\sum_{k=1}^n \sigma_2^k(D_{s-},s)
                                   \sigma_1^k(D_{s-},s)\Bigr) ds  \notag \\
                             &+\int_0^{E_t} \Phi_{D(s-)}^{-1} \sum_{k=1}^n\sigma_1^k(D_{s-},s) dB_s^k \Biggr], \notag
     \end{align}
     where $\Phi=(\Phi_t)$ is the fundamental solution to the homogeneous linear SDE corresponding to \eqref{SDE38}.  Namely, 
     $\Phi$ is the unique $d\ttimes d$-matrix-valued process satisfying the homogeneous SDE
     \begin{equation}
          d\Phi_t=\rho_2(t,E_t)\Phi_t dt+\mu_2(t,E_t)\Phi_t dE_t 
                +\sum_{k=1}^n \sigma_2^k(t,E_t)\Phi_t dB_{E_t}^k, \label{SDE40.1} \\
     \end{equation}
     with initial condition $\Phi_0=I_d$, where $I_d$ denotes the $d\ttimes d$-identity matrix. 
\end{Cor}

\begin{proof}   
     We first claim that for each path, $\Phi_t$ is invertible for all $t\geq 0$.  Otherwise, there would exist 
     $t_0\geq 0$ and $\lambda\in\bbR^d\setminus \{0\}$ such that $\Phi_{t_0}\lambda=0$.  
     The $d$-dimensional process $(\Phi_t \lambda)$ satisfies the homogeneous linear SDE 
     \begin{align}
          d\Psi_t=\rho_2(t,E_t)\Psi_t dt+\mu_2(t,E_t)\Psi_t dE_t 
                +\sum_{k=1}^n \sigma_2^k(t,E_t)\Psi_t dB_{E_t}^k. \label{SDE40.2}                 
     \end{align}
     The zero process is 
     the unique solution to \eqref{SDE40.2} for which $\Psi_{t_0}=0\in\bbR^d$.  
     Therefore, it follows that $\Phi_t\lambda=0$ for all $t\geq 0$, which contradicts $\Phi_0\lambda=\lambda\neq 0$. 
     Thus, $\Phi_t$ is invertible for all $t\geq 0$, and the right hand side of SDE \eqref{SDE39} is meaningful. 
     
     As in the proof of Proposition \ref{Prop LSDE-2}, it suffices to show that $X$ given in \eqref{SDE39}
     satisfies SDE \eqref{SDE38}.  Using the calculus rule 
     \begin{align*}
          [m,B^k\circ E]=[E,B^k\circ E]=0 \ \ \ \textrm{and} \ \ \ [B^k\circ E,B^\ell\circ E]=\delta^{k,\ell}E,
     \end{align*}
     where $\delta^{k,\ell}$ is the Kronecker delta, and applying the It\^o formula componentwise, 
     the proof is carried out in the same way as in Theorem \ref{Thm LSDE-3}.  
\end{proof}

\begin{Rems}\label{Rems LSDE-solution}\par
\begin{em}
     (a) The advantage of rewriting solutions in the forms \eqref{SDE24} and \eqref{SDE32} is that 
     they can be handled via conditioning on the random variable $E_t$.  
     This is especially useful in analyzing statistical data of a solution, such as its mean and variance.  
     If $Y$ is the solution to a classical It\^o SDE with linear coefficients 
     $dY_t=\bigl(b_1(t)+b_2(t)Y_t\bigr)dt +\bigl(\tau_1(t)+\tau_2(t)Y_t\bigr)dB_t$, 
     then the first two moments of $Y_t$ are characterized as solutions to linear ordinary differential equations (ODEs), 
     from which some information on statistics can be derived. 
     (See \cite[Thm.\ 4.5]{Gard} for a general case, or \cite[Problem 5.6.1]{Karatzas-S} 
     for a special case when $\tau_2\equiv 0$.) 
     However, it is generally impossible to obtain such ODEs for the solution $X$ to SDE \eqref{SDE21}, 
     even when $\rho_j$, $\mu_j$, $\sigma_j$ are deterministic. 
     For example, consider the SDE $dX_t=\mu_2(t)X_t dE_t$, a special case of \eqref{SDE21}.  
     Taking expectations in the integral form, $\bbE[X_t]=x_0+\bbE[\int_0^t \mu_2(s)X_s dE_s]$.  
     The expectation and integral are not interchangeable due to the presence of the random integrator $dE_s$. 
     As a result, unlike the case of a classical It\^o SDE, a general form of an ODE satisfied by $\bbE[X_t]$ cannot be obtained.  
     This observation heightens the importance of expressions such as \eqref{SDE24} and \eqref{SDE32}.   
     Moreover, since these expressions are derived via the 2nd change-of-variable formula \eqref{COV1}, 
     there is no doubt that Formula \eqref{COV1} is an indispensable tool for dealing with SDEs of the form \eqref{SDE1}.
   
     (b) {\it Recognizing how  our new class of SDEs of the form  \eqref{SDE21} arise: Viewpoint 1.} 
     If $E_t=t$ and $\rho_j$, $\mu_j$, $\sigma_j$ $(j=1,2)$ are all deterministic, then  
     Proposition \ref{Prop LSDE-2} and Theorem \ref{Thm LSDE-3} respectively reduce to well-known results for classical It\^o SDEs 
     with linear coefficients 
     \begin{align}
          dY_t=\bigl(b_1(t)+b_2(t)Y_t\bigr)dt +\bigl(\sigma_1(t)+\sigma_2(t)Y_t\bigr)dB_t, \label{SDE36}
     \end{align}
     where $b_j(t)=\rho_j(t)+\mu_j(t)$ $(j=1,2)$. (See \cite[Thm.\ 4.2]{Gard}.) 
     This observation suggests that 
     an SDE of the form \eqref{SDE21} might be constructed 
     via continuously altering the clock from $t$ to $E_t$ in \eqref{SDE36},
     but with the drift factor $b_j$ splitting into two components $\rho_j$ and $\mu_j$, 
     the former reflecting the effect of the original clock $t$ 
     and the latter of the new clock $E_t$.   
     Allocation of the weight of $b_j$ to $\rho_j$ and $\mu_j$ is due to consideration of 
     how much the time-changed model is affected by the new clock.  
     If the absolute value of $\rho_j$ is big (resp.\ small) in comparison to that of $\mu_j$, then the model \eqref{SDE21} 
     contains a large (resp.\ small) effect of the original clock.    
     SDE \eqref{SDE3} with $Z=B$ provides an example where there is no effect of the original clock.  
     Note that $\rho_j$ and $\mu_j$ may take negative values as well.  
     
     {\it Viewpoint 2.}  Again assume $\mu_j$, $\sigma_j$ $(j=1,2)$ are all deterministic.  Adopt a classical It\^o SDE 
     \begin{align}
          dZ_t=\bigl(\rho_1(t)+\rho_2(t)Z_t\bigr)dt +\bigl(\tau_1(t)+\tau_2(t)Z_t\bigr)dB_t \label{SDE37}
     \end{align}
     as the starting form of SDE \eqref{SDE21}.     
     This interpretation is valid 
     when path properties or statistical data of the solution to a simple SDE of the form \eqref{SDE37} 
     fail to match the real data (especially in terms of the volatility coefficients $\tau_j$), 
     but clearly possesses a drift similar to $\bigl(\rho_1(t)+\rho_2(t)Z_t\bigr)dt$.  
     In this situation, one prefers to `break' the $dB_t$ term via changing the clock from $t$ to $E_t$ 
     so that the model has more flexibility in describing the volatility. 
     As a result, $dE_t$ and $dB_{E_t}$ terms are obtained as in \eqref{SDE21}, without changing the drift coefficients $\rho_j$.      
     Note that the arguments from both viewpoints apply to a general class of SDEs of the form \eqref{SDE1} as well. 
   
     (c) The general form of solutions obtained in Proposition \ref{Prop LSDE-2}, Theorem \ref{Thm LSDE-3} and Corollary \ref{Cor LSDE-3} 
     are all valid even when SDE \eqref{SDE21} has general process coefficients.  
     More precisely, if the coefficients $\rho_j$, $\mu_j$, $\sigma_j$ $(j=1,2)$ are processes in $\bbL\cG$ 
     with $\mathcal{G}_t\DEF \mathcal{F}_{E_t}$ such that 
     the absolute values of $\rho_2$, $\mu_2$, $\sigma_2$ are dominated by some random variable $L$, then  
     it can be shown by reformulating Theorem 7 of \cite[Chap.\ V]{Protter} that 
     SDE \eqref{SDE21}, with the coefficients evaluated at $(t,\omega)$ rather than $\bigl(t,E_t(\omega)\bigr)$, 
     has a unique strong solution; moreover, 
     the explicit form of the solution has exactly the same expression as in the previous results.   
\end{em}
\end{Rems}

     Just as there is a reduction method for classical It\^o SDEs with nonlinear coefficients  
     \begin{align}
          dY_t=b(t,Y_t)dt+\tau(t)Y_tdB_t \ \ \ \textrm{with} \ Y_0=x_0, \label{SDE91}
     \end{align}
     Proposition \ref{Prop Reduction} provides an analogous technique for approaching a certain type of nonlinear SDE 
     including terms driven by a time-changed Brownian motion.  
     The `integrating factor' 
     \begin{align*}
          U_t\DEF \exp\Bigl\{ \frac 12 \int_0^t \tau^2(s)ds-\int_0^t \tau(s) dB_s \Bigr\}
     \end{align*} 
     reduces \eqref{SDE91} to a path-by-path ODE $d(U_tY_t) =U_t\cdot b(t,Y_t)dt$, with $U_0Y_0=x_0$,
     computation of which almost traces the proof of Proposition \ref{Prop Reduction}.    
     Applications of this reduction scheme are provided in Examples \ref{Ex LSDE-9} and \ref{Ex LSDE-10}.

\begin{Prop}\label{Prop Reduction}\par \hspace{-2mm}
     \textbf{\emph{(Reduction Method)}} \
     Let $B$ be a standard $\cF$-Brownian motion. 
     Let $E$ be a continuous $\cF$-time-change.
     Then the `integrating factor' $U$ defined by 
     \begin{align}
          U_t\DEF \exp\Biggl\{ \int_0^t \Bigl(\frac 12 \sigma_2^2(s,E_s)-\mu_2(s,E_s) \Bigr) dE_s-\int_0^t \sigma_2(s,E_s) dB_{E_s}
               \Biggr\} \label{SDE92}
     \end{align}
     reduces the nonlinear SDE 
     \begin{equation}
          dX_t=\rho(t,E_t,X_t)dt+\mu_2(t,E_t)X_tdE_t+\sigma_2(t,E_t)X_tdB_{E_t}
             \ \ \ \textrm{with} \ X_0=x_0, \label{SDE93}
     \end{equation}
     to a path-by-path ODE
     \begin{align}
          \frac{dW_t}{dt}=U_t\cdot \rho\bigl(t,E_t,U_t^{-1}W_t\bigr)
          \ \ \ \textrm{with} \ W_0=x_0, \label{SDE94}
     \end{align}
     where $W_t\DEF U_tX_t$ so that $X_t=U_t^{-1}W_t$.  
\end{Prop}

\begin{proof}   
     For notational convenience, we suppress the dependence on $E_t$ and simply write 
     $\rho(t,X_t)=\rho(t,E_t,X_t)$, $\mu_2(t)=\mu_2(t,E_t)$ and $\sigma_2(t)=\sigma_2(t,E_t)$.
     Write $U_t=e^{A_t}$ so that 
     \begin{align*}
          A_t=\int_0^t \Bigl(\frac 12 \sigma_2^2(s)-\mu_2(s) \Bigr) dE_s-\int_0^t \sigma_2(s) dB_{E_s}.
     \end{align*} 
     Then the It\^o formula \eqref{ItoFormula1} with $f(a)=e^a$ together with the calculus rules in  \eqref{SDE20} yield 
     \begin{align*}
          dU_t=U_t dA_t+\tfrac 12 U_t d[A,A]_t=U_t\bigl\{\bigl(\sigma_2^2(t)-\mu_2(t)\bigr)dE_t-\sigma_2(t)dB_{E_t}\bigr\}. 
     \end{align*}
     Hence, by the product formula \eqref{ItoFormula2},
     \begin{align}
          d(U_tX_t)&=U_t dX_t+X_tdU_t+d[U,X]_t\label{SDE95}\\
                      &=U_t \bigl\{\rho(t,X_t)dt+\mu_2(t)X_tdE_t+\sigma_2(t)X_tdB_{E_t} \bigr\}\notag\\
                      & \ \ +X_tU_t \bigl\{\bigl(\sigma_2^2(t)-\mu_2(t)\bigr)dE_t-\sigma_2(t)dB_{E_t}
                           \bigr\} - \sigma_2^2(t)X_tU_tdE_t\notag\\
                      &=U_t\cdot \rho(t,X_t)dt. \notag
     \end{align}
     By setting $W_t\DEF U_tX_t$, \eqref{SDE95} immediately yields \eqref{SDE94}.      
\end{proof}

\section{Examples} \label{sec example}

     The examples below are drawn from the classical It\^o SDEs; however, the driving processes
     involve a continuous time-change $E$ and the time-changed Brownian motion $B\circ E$.  
     Assume that all coefficients of SDEs appearing in this section satisfy the conditions \eqref{SDE2} and \eqref{SDE2.5}.

\begin{Ex}\label{Ex LSDE-4}\par
\begin{em}
     The most basic linear SDE is the homogeneous one with constant coefficients, 
     which is an analogue of the so-called Black--Scholes SDE.  Consider
     \begin{align}
          dX_t=\rho X_tdt+\mu X_tdE_t+\sigma X_tdB_{E_t} \ \ \ \textrm{with} \ X_0= x_0, \label{SDE41}
     \end{align}
     where $\rho$, $\mu$, $\sigma$ are real constants and $x_0>0$, $\sigma>0$.  
     
     The case where $E_t=t$ corresponds to the Black--Scholes model $dY_t=b Y_tdt+\sigma Y_tdB_t$ with $Y_0=x_0$,
     where $b=\rho+\mu$. 
     The solution
     \begin{align*}
          Y_t=x_0 \exp\Bigl\{\Bigl(b-\frac 1 2 \sigma^2 \Bigr)t+\sigma B_t \Bigr\}
     \end{align*}
     has the following asymptotic behavior: 
     \begin{enumerate}[({Y.}1)]
          \item If $b>\sigma^2/2$, then $\lim_{t\to\infty}Y_t=\infty$.  
          \item If $b<\sigma^2/2$, then $\lim_{t\to\infty}Y_t=0_+$.  
          \item If $b=\sigma^2/2$, then $Y_t$ asymptotically fluctuates 
                           between arbitrarily large and arbitrarily small positive values infinitely often.   
     \end{enumerate}
     This follows by rewriting the solution as $Y_t=x_0 \exp\bigl\{t\bigl[(b-\sigma^2/2)+\sigma \cdot B_t/t \bigr]\bigr\}$ 
     and using the law of the iterated logarithm for paths of Brownian motion
     \begin{align}
          \limsup_{t\to\infty} \dfrac{B_t}{\sqrt{2t \log\log t}}=1 \ \ \ \textrm{and} 
          \ \ \ \liminf_{t\to\infty} \dfrac{B_t}{\sqrt{2t \log\log t}}=-1. \label{SDE43}
     \end{align}     
     For details of this classical model, consult \cite{Steele}.

     Analysis of the asymptotic behavior of the solution to SDE \eqref{SDE41} is accomplished 
     with the help of the explicit solution form obtained from \eqref{SDE23},
     \begin{align}
          X_t=x_0 \exp\Bigl\{\rho t+\Bigl(\mu-\frac 12 \sigma^2\Bigr)E_t+\sigma B_{E_t}\Bigr\}.\label{SDE42}
     \end{align} 
 
     First, if $\rho=0$, i.e., if there is no effect of the original clock upon the solution of SDE \eqref{SDE41}, 
     then, by Theorem \ref{Thm SDE-duality} (2), $(X_{D_t})$ satisfies the classical It\^o SDE 
     $dX_{D_t}=\mu X_{D_t}dt+\sigma X_{D_t}dB_t$.  
     Since $\lim_{t\to\infty} D_t=\infty$, $X$ has the same asymptotic behavior as the above-mentioned $Y$ with $b$ replaced by $\mu$:
     \begin{enumerate}[({X.a.}1)]
          \item If $\rho=0$ and $\mu>\sigma^2/2$, then $\lim_{t\to\infty}X_t=\infty$.  
          \item If $\rho=0$ and $\mu<\sigma^2/2$, then $\lim_{t\to\infty}X_t=0_+$.  
          \item If $\rho=0$ and $\mu=\sigma^2/2$, then $X_t$ asymptotically fluctuates 
                           between arbitrarily large and arbitrarily small positive values infinitely often.   
     \end{enumerate}
 
     Next, suppose $\rho\neq 0$.  
     Assume $\lim_{t\to\infty}E_t=\infty$ and $\lim_{t\to\infty}E_t/t=0$; i.e., $E_t$ is asymptotically slower than $t$.  
     By rewriting \eqref{SDE42} as 
     \begin{align*}
          X_t=x_0 \exp\Biggl\{t\Biggl[ \rho+\Bigl(\mu-\frac 12 \sigma^2\Bigr)\frac{E_t}{t}
              +\sigma \cdot \frac{B_{E_t}}{E_t}\cdot \frac{E_t}{t} \Biggr]\Biggr\}
     \end{align*} 
     and using \eqref{SDE43} again, we easily observe 
     \begin{enumerate}[({X.b.}1)]
          \item If $\rho>0$ and $E_t$ is asymptotically slower than $t$, then $\lim_{t\to\infty}X_t=\infty$.  
          \item If $\rho<0$ and $E_t$ is asymptotically slower than $t$, then $\lim_{t\to\infty}X_t=0_+$. 
     \end{enumerate}
     These cases match with our intuition: if the original clock $t$ asymptotically ticks more frequently than the new clock $E_t$, 
     then the $\rho$ describing the effect of the original clock completely determines the future behavior of 
     the solution $X$, no matter what values $\mu$ and $\sigma$ take.       

     On the other hand, if $E_t$ grows faster than $t$, i.e.,  if $\lim_{t\to\infty}E_t/t=\infty$, 
     then the situation becomes much more complicated. 
     Rewrite \eqref{SDE42} as 
     \begin{align*}
         X_t=x_0 \exp\Biggl\{E_t\Biggl[ \rho\frac{t}{E_t}+\Bigl(\mu-\frac 12 \sigma^2\Bigr)+\sigma \frac{B_{E_t}}{E_t} \Biggr]\Biggr\}. 
     \end{align*} 
     By noting \eqref{SDE43} again, we observe 
     \begin{enumerate}[({X.c.}1)]
          \item If $\rho\neq 0$, $\mu>\sigma^2/2$ and $E_t$ grows faster than $t$, then $\lim_{t\to\infty}X_t=\infty$.  
          \item If $\rho\neq 0$, $\mu<\sigma^2/2$ and $E_t$ grows faster than $t$, then $\lim_{t\to\infty}X_t=0_+$. 
          \item If $\rho\neq 0$, $\mu=\sigma^2/2$ and $E_t$ grows faster than $t$, then the fluctuation of $X_t$ 
                   varies depending on the coefficients of the SDE and also the speed at which $E_t$ grows.  
     \end{enumerate}
     The first two cases show that if $\mu\neq \sigma^2/2$ and $E_t$ grows faster than $t$, then 
     the asymptotic behavior of $X$, regardless of the value of $\rho(\neq 0)$, coincides with (X.a.1) and (X.a.2). 
     This is due to the fact that the effect of the faster clock $E_t$ is strongly reflected on $\mu$ 
     to the extent that $\rho$ is ignored.   

     In the special situation (X.c.3), if $\lim_{t\to\infty}\sqrt{2E_t \log\log E_t}/t=\infty$ so that $E_t$ grows extremely fast, then 
     $X_t$ asymptotically takes arbitrary values on the positive real line infinitely many times.  
     This is immediate upon writing 
     \begin{align*}
          X_t=x_0 \exp\Biggl\{t\Biggl[ \rho+\sigma \cdot \frac{B_{E_t}}{\sqrt{2E_t \log\log E_t}} 
            \cdot \frac{\sqrt{2E_t \log\log E_t}}{t} \Biggr]\Biggr\}.
     \end{align*}   
     On the other hand, if, e.g., $\lim_{t\to\infty}\sqrt{2E_t \log\log E_t} /t=0$, then  
     $X_t$ asymptotically goes off to $\infty$ if $\rho>0$ and decreases to $0$ if $\rho<0$. 
     
     These observations establish that as the time-change $E$ accelerates the speed at which time passes, 
     dependence of the behavior of the solution $X$ upon $\rho$ and $\mu$ respectively becomes lighter and heavier. 
     \qed
\end{em}
\end{Ex}

\begin{Ex}\label{Ex LSDE-6}\par
\begin{em}
     Assume $B$ is independent of $D$, or equivalently, of $E$.  The homogeneous linear SDE 
     \begin{align}
          dX_t=\rho(t) X_tdt+\mu(E_t)X_tdE_t+\sigma(E_t) X_tdB_{E_t} \ \ \  \textrm{with} \ X_0= x_0,  \label{SDE61}
     \end{align}
     where $x_0> 0$, has a unique strong solution $X$ expressed as \eqref{SDE24}.  
     
     The value of the mean function $\bbE[X_t]$ can be investigated by
     conditioning on $E_t$ and using the independence of $B$ and $E$:
     \begin{align}
          \bbE[X_t]
          &= x_0\exp\Bigl\{ \int_0^t \rho(s)ds \Bigr\} \label{SDE62}
          \cdot  \bbE \Bigl[ \exp\Bigl\{\int_0^{E_t}\Bigl( \mu(s)-\dfrac 12 \sigma^2(s) \Bigr) ds 
                 +\int_0^{E_t} \sigma(s) dB_s\Bigr\}\Bigr] \\
          &= x_0\exp\Bigl\{ \int_0^t \rho(s)ds \Bigr\} \notag\\
          & \ \ \ \times \int_0^\infty \bbE \Bigl[ \exp\Bigl\{\int_0^v \Bigl( \mu(s)-\dfrac 12 \sigma^2(s) \Bigr) ds 
                 +\int_0^v \sigma(s) dB_s\Bigr\}\Bigr] \hspace{1mm} p_t(dv) \notag \\
          &= x_0\exp\Bigl\{ \int_0^t \rho(s)ds \Bigr\} \cdot \int_0^\infty \exp\Bigl\{\int_0^v \mu(s) ds \Bigr\} 
                 \cdot \bbE [M_v] \hspace{2mm} p_t(dv), \notag 
     \end{align}
     where $p_t$ denotes the law of the random variable $E_t$ and $M$ is a continuous $\cF$-local martingale given by 
     \begin{align}
          M_v\DEF \exp\Bigl\{-\dfrac 12\int_0^v \sigma^2(s) ds +\int_0^v \sigma(s) dB_s\Bigr\}. \label{SDE63}
     \end{align}
     Actually the process $M$ is a martingale since $\sigma$ satisfies the Novikov condition; i.e., \linebreak 
     $\bbE[\exp\{\frac 12 \int_0^v \sigma^2(s)ds\}]<\infty$ for all $v\geq 0$. 
     (See \cite[Prop.\ 3.5.12]{Karatzas-S}.)
     Hence, $\bbE[M_v]=1$ for all $v\geq 0$. 
     Thus, \eqref{SDE62} yields 
     \begin{align}
          \bbE[X_t]=x_0\exp\Bigl\{ \int_0^t \rho(s)ds \Bigr\}\cdot \int_0^\infty \exp\Bigl\{\int_0^v \mu(s) ds\Bigr\} 
              \hspace{1mm} p_t(dv). \label{SDE64}
     \end{align}
     If $E_t=t$, then $p_t=\delta_t$, the Dirac measure with mass at $t$.  Hence, \eqref{SDE64} yields 
     $\bbE[X_t]=x_0\exp\bigl\{ \int_0^t \bigl(\rho(s)+ \mu(s) \bigr)ds\bigr\}$, 
     which, of course, coincides with the mean function $\bbE[Y_t]$ of the solution $Y$ to the classical It\^o SDE
     $dY_t=\bigl(\rho(t)+\mu(t)\bigr) Y_tdt+\sigma(t) Y_tdB_t$ with $Y_0= x_0$.   
     The result \eqref{SDE64} shows that the behaviors of $\rho$ and $\mu$ together
     govern the range of fluctuation of the mean function $\bbE[X_t]$.  
     Moreover, even when the coefficient of $dE_t$ in SDE \eqref{SDE61} is replaced by a more general $\mu(t,E_t)X_t$, 
     some form of estimate on $\bbE[X_t]$ can still be obtained.  For instance, if $\int_0^v \mu(D_{s-},s)ds\geq 0$ for all $v\geq 0$, 
     then $\bbE[X_t]\geq x_0 \exp\bigl\{ \int_0^t \rho(s)ds \bigr\}$.  

     The variance function $\bbV[X_t]$ of the solution $X$ is computed similarly:
     \begin{align*}  
        \bbV[X_t] 
         =\ &x_0^2\exp\Bigl\{ 2\int_0^t \rho(s)ds \Bigr\} 
         \cdot \Biggl[\int_0^\infty \exp\Bigl\{2\int_0^v \mu(s) ds+\int_0^v \sigma^2(s)ds \Bigr\} \hspace{1mm} p_t(dv) \\ 
         &- \Biggl( \int_0^\infty \exp\Bigl\{\int_0^v \mu(s) ds \Bigr\} \hspace{1mm} p_t(dv)\Biggr)^{\hspace{-1mm}2\hspace{1mm}} \Biggr].\notag 
     \end{align*}
     Unlike the explicit form of the mean function in \eqref{SDE64}, $\bbV[X_t]$ involves the information $\sigma$
     concerning the weight of the $dB_{E_t}$ term in SDE \eqref{SDE61}. 
     
     As a special case of SDE \eqref{SDE61}, 
     assume $\mu(u)\equiv -\lambda$ for some $\lambda>0$.  Then \eqref{SDE64} is expressed 
     in terms of the Laplace transform of the law of $E_t$:
     \begin{align}
          \bbE[X_t]=x_0\exp\Bigl\{ \int_0^t \rho(s)ds \Bigr\}\cdot \int_0^\infty e^{-\lambda v} \hspace{1mm} p_t(dv). \label{SDE66}
     \end{align}
     Moreover, if $E$ is the first hitting time process of an $\cF$-stable subordinator of index $\beta\in(0,1)$ which is independent of $B$, 
     then the Laplace transform in \eqref{SDE66} 
     is associated with the Mittag--Leffler function due to \cite[Thm.\ 4.3]{BKS}:  
     \begin{align}
          \bbE[X_t]=x_0\exp\Bigl\{ \int_0^t \rho(s)ds \Bigr\}\cdot \textbf{E}_\beta(-\lambda t^\beta), \label{SDE67}
     \end{align}
     where $\textbf{E}_\beta(z)\DEF \sum_{n=0}^\infty z^n/\Gamma(\beta n+1)$ with $\Gamma(\cdot)$ being the Gamma function.  \qed  
\end{em}
\end{Ex}

\begin{Ex}\label{Ex LSDE-7}\par
\begin{em}
     Consider the inhomogeneous linear SDE 
     \begin{align}
          dX_t=\ &\Bigl(\dfrac{b}{1-t}-\dfrac{\gamma}{1-t} X_t\Bigr)dt+\Bigl(\dfrac{c}{1-E_t}-\dfrac{\eta}{1-E_t} X_t\Bigr) dE_t 
            +dB_{E_t}, \ t\in[0,1), \label{SDE71} \\ 
            & \ \textrm{with} \ X_0= a, \notag
     \end{align}
     where $a$, $b$, $c$, $\gamma$, $\eta \in\bbR$ and $E_t$ increases to $1$ as $t$ increases to $1$.  
     
     The fundamental solution to the homogeneous linear SDE corresponding to \eqref{SDE71} is $\Phi_t=(1-t)^\gamma(1-E_t)^\eta$.  
     Hence, \eqref{SDE31} yields  
     \begin{align} 
          X_t=\ &(1-t)^\gamma(1-E_t)^\eta a + \int_0^t \frac{b}{1-s}
                \Bigl(\frac{1-t}{1-s}\Bigr)^\gamma\Bigl(\frac{1-E_t}{1-E_s}\Bigr)^\eta ds \label{SDE72}\\
               &+ \int_0^t \frac{c}{1-E_s}\Bigl(\frac{1-t}{1-s}\Bigr)^\gamma\Bigl(\frac{1-E_t}{1-E_s}\Bigr)^\eta dE_s
               + \int_0^t \Bigl(\frac{1-t}{1-s}\Bigr)^\gamma\Bigl(\frac{1-E_t}{1-E_s}\Bigr)^\eta dB_{E_s}. \notag
     \end{align}         
     If $E_t=t$, then the solution \eqref{SDE72} reduces to 
     \begin{align} 
          X_t= (1-t)a-t(b+c)+(1-t)\int_0^t \frac{1}{1-s}dB_s,
     \end{align}         
     a Brownian bridge from $a$ to $(b+c)$.  
     Moreover, the class of SDEs of the form \eqref{SDE71} contains a `time-changed Brownian bridge' from $a$ to $c$.  
     In fact, if $b=0$, $\gamma=0$ and $\eta=1$, 
     then $X$ satisfies the SDE
     \begin{align}
          dX_t=\Bigl(\dfrac{c}{1-E_t}-\dfrac{1}{1-E_t} X_t\Bigr) dE_t +dB_{E_t}, \ t\in[0,1), \ \ \ \textrm{with} \ X_0= a,
     \end{align}
     which is, by Theorem \ref{Thm SDE-duality}, associated with the classical Brownian bridge SDE
     \begin{align}
          dY_t=\Bigl(\dfrac{c}{1-t}-\dfrac{1}{1-t} Y_t\Bigr) dt +dB_t, \ t\in[0,1), \ \ \ \textrm{with} \ Y_0= a, 
     \end{align}
     via the relation $X=Y\circ E$.  Thus, in this particular case, $X$ is a process obtained by time-changing the Brownian bridge $Y$.
     \qed
\end{em}
\end{Ex}

     Viewpoint 1 of Remark \ref{Rems LSDE-solution} (b) 
     states that it is possible 
     to recognize that the two components $\rho_j$ and $\mu_j$ of SDE \eqref{SDE21} 
     are produced by splitting the drift factor of some classical It\^o SDE.   
     Examples \ref{Ex LSDE-4}, \ref{Ex LSDE-6} and \ref{Ex LSDE-7} are all discussed from this viewpoint.  
     However, as mentioned in Viewpoint 2 of the remark, 
     it is also possible to attribute the presence of $\mu_j$ to the $dB_t$ term in a classical It\^o SDE.  
     Example \ref{Ex LSDE-8} illustrates this viewpoint.

\begin{Ex}\label{Ex LSDE-8}\par
\begin{em}
     This example investigates statistical data obtained from the solution to the inhomogeneous linear SDE
     \begin{align}
          dX_t=-\alpha X_tdt +\mu dE_t+\sigma dB_{E_t} \ \ \ \textrm{with} \ X_0=x_0, \label{SDE81}
     \end{align}
     where $\alpha$, $\sigma>0$, $\mu\in\bbR$, and $x_0\neq 0$.  
     SDE \eqref{SDE81} with $E_t=t$ and $\mu=0$ is called the Langevin equation or the Ornstein--Uhlenbeck model, 
     and its solution is referred to as the Ornstein--Uhlenbeck process.   
     The coefficient $-\alpha X_t$ of the $dt$ term is negative (resp.\ positive) when $X_t$ is positive (resp.\ negative), which implies 
     $X_t$ is drawn back to zero once it drifts away.  
     Since the coefficient $\mu$ describing the drift based on the new clock $E_t$ is not proportional to the current position $X_t$,   
     if, e.g., $E_t$ represents the business time at the calendar time $t$, then 
     $X_t$, regardless of its value, is always affected by the evolution of the business time.  
     In other words, the model has a certain factor of weight $\mu$ which pushes the position either up or down during business hours,  
     and its effect on the position becomes larger (resp.\ smaller) when the business time grows faster (resp.\ slower).     
     Moreover, the dispersion coefficient $\sigma$ does not depend on the position either.   
     Therefore, unless the time-change $E$ either stays flat or accelerates or decelerates drastically on an interval, 
     $X_t$ fluctuates on this interval at a certain rate with mild error even when it approaches close to zero. 
     In finance, the Ornstein--Uhlenbeck-type model \eqref{SDE81}, which incorporates a possible time-change,  
     could be used to describe the deviation of an interest rate around a central bank's target rate.

     Assume both of the following technical conditions are satisfied: 
     \begin{enumerate}[(a)]
          \item for each $t\ge 0$, the random variable $E_t$ is bounded; i.e., $\bbP(E_t\leq c_t)=1$ for some finite positive constant $c_t$;   \vspace{1mm}
          \item $\bbE\bigl[ \int_0^t e^{2\alpha D_{s-}} ds \bigr]<\infty$ for all $t\geq 0$.            
     \end{enumerate}  
     The monotonicity of $D$ implies that the condition (b) is equivalent to: 
     \begin{enumerate}[(a')]
          \item[(b')] $\bbE\bigl[ e^{2\alpha D_{t-}} \bigr]<\infty$ for all $t\geq 0$.            
     \end{enumerate}  
    
     Let us analyze the mean $\bbE[X_t]$ of the solution $X$ to SDE \eqref{SDE81}.  
     By \eqref{SDE31} and 
     \eqref{SDE32}, $X$ can be represented in two ways:
     \begin{align}
          X_t&=e^{-\alpha t} \Bigl\{ x_0+\mu \int_0^t e^{\alpha s} dE_s +\sigma \int_0^t e^{\alpha s} dB_{E_s}\Bigr\}\label{SDE82}\\
             &=e^{-\alpha t} \Bigl\{ x_0+\mu \int_0^{E_t} e^{\alpha D_{s-}} ds +\sigma \int_0^{E_t} e^{\alpha D_{s-}} dB_s\Bigr\}.\notag
     \end{align}   
     By assumption (b), the process $N$ defined by $N_t\DEF \int_0^t e^{\alpha D_{s-}} dB_s$ is an $\cF$-martingale. 
     Since each $E_t$ is a bounded $\cF$-stopping time due to (a), Doob's optional sampling theorem  
     yields $\bbE[N_{E_t}]=\bbE[N_0]=0$.  (See \cite[Problem 1.3.23 (i)]{Karatzas-S}.) 
     Hence, taking expectations in \eqref{SDE82}, 
     \begin{align}
          \bbE[X_t]= e^{-\alpha t} \Biggl\{ x_0+\mu \hspace{1mm}\bbE\Bigl[\int_0^t e^{\alpha s} dE_s\Bigr] \Biggr\}\label{SDE83}
                   = e^{-\alpha t} \Biggl\{ x_0+\mu \hspace{1mm}\bbE\Bigl[\int_0^{E_t} e^{\alpha D_{s-}} ds\Bigr] \Biggr\}.
     \end{align}
     Consequently, the asymptotic behavior of the mean function $\bbE[X_t]$ completely 
     depends on the distributions of the processes $E$ and $D$.     
     In the special case where $E_t(\omega)=R(\omega)\cdot t$ for some positive random variable $R$, 
     $\bbE[X_t]=x_0\, e^{-\alpha t}+(\mu \bbE[R]/\alpha)(1-e^{-\alpha t})$, 
     which approaches $\mu \bbE[R]/\alpha$ as $t\to\infty$. 
     Therefore, if the force attracting $X_t$ to zero is sufficiently strong compared to 
     the factor producing the effect of the evolution of the time 
     (i.e., if $\alpha$ is much larger than the absolute value of $\mu$ and $\bbE[R]$), 
     then the expected value of the position tends to a level close to zero as $t\to\infty$.  
     On the other hand, the bigger the weight $\mu$ or the expected rate $\bbE[R]$ of acceleration of the new clock, 
     the greater the asymptotic value of the expected position.  
     
     Another way to observe the fluctuation of $\bbE[X_t]$
     is to directly analyze the integral form of the SDE \eqref{SDE81}. 
     Taking the expectation, 
     \begin{align}
          \bbE[X_t]=-\alpha \int_0^t \bbE[X_s] ds+\mu\bbE[E_t]+\sigma\bbE[B_{E_t}].  \label{SDE84}
     \end{align}
     The last term vanishes again due to the assumption (a) and Doob's optional sampling theorem.  
     Hence, we obtain a differential equation
     \begin{align}
          \frac d {dt}\bbE[X_t]=-\alpha \bbE[X_t]+\mu \frac d{dt}\bbE[E_t]
          \ \ \ \textrm{with} \ \bbE[X_0]=x_0, \ \bbE[E_0]=0. \label{SDE85}
     \end{align}
     Although this is not the explicit form of $\bbE[X_t]$ obtained in \eqref{SDE83}, 
     it still provides information on the relationship between the time evolutions of $\bbE[X_t]$ and $\bbE[E_t]$.  
     
     The term $\bbE[B_{E_t}]$ in \eqref{SDE84} vanishes even when the assumption (a) is replaced by one of the following:
     \begin{enumerate}[(a)]
          \item[(c)] $\bbE[\sqrt{E_t}\,]<\infty$ for all $t\geq 0$;   \vspace{1mm}
          \item[(d)] $B$ is independent of $E$.            
     \end{enumerate}   
     If condition (c) holds, which is weaker than (a), then the `Wald identity' 
     $\bbE[B_{E_t}]=0$ holds for each $t\geq 0$.
     (See \cite[Problem 3.2.12, Exercise 3.3.35]{Karatzas-S}.) 
     On the other hand, (d) encourages conditioning on the random variable $E_t$ to obtain $\bbE[B_{E_t}]=0$.  
     
     Suppose $E$ is the first hitting time process of an $\cF$-stable subordinator of index $\beta\in (0,1)$
     which is independent of $B$, so condition (d) holds by assumption.
     There is a positive constant $c(\beta)$ such that $\bbE[E_t]=c(\beta)\, t^{\beta}$ for all $t\geq 0$, due to \cite[Cor.\ 3.1]{M-S_1}.
     Hence, (c) also holds. 
     Moreover, using this moment result, \eqref{SDE85} is reexpressed as 
     \begin{align}
          \frac d {dt}\bbE[X_t]=-\alpha \bbE[X_t]+\mu \beta \hspace{1.5pt}c(\beta)\, t^{\beta-1} 
          \ \ \ \textrm{with} \ \bbE[X_0]=x_0. \label{SDE86}
     \end{align}      
     The solution of the first order linear ODE \eqref{SDE86} is given by 
     \begin{align}
          \bbE[X_t]&=e^{-\alpha t}\Bigl\{ x_0+\mu\beta \hspace{1.5pt} c(\beta) \int_0^t e^{\alpha s} s^{\beta-1} ds \Bigr\} \label{SDE87}\\ 
              &=e^{-\alpha t}\Bigl\{ x_0+\mu\beta \hspace{1.5pt}c(\beta) \int_0^t g_{\alpha,t}(r) (t-r)^{\beta-1} dr \Bigr\} \notag\\
              &=e^{-\alpha t}\Bigl\{ x_0+\mu\beta 
                     \hspace{1.5pt}c(\beta)\hspace{1.5pt} \Gamma(\beta) \cdot (J^\beta g_{\alpha,t})(t)\Bigr\},\notag
     \end{align}     
     where $g_{\alpha,t}(r)\DEF e^{\alpha(t-r)}$, and $\Gamma(\cdot)$ and $J^\beta$ respectively denote the Gamma function and  
     the fractional integral of order $\beta$. (For definition of fractional integrals, see \cite{G-M}.)
     
     An interesting conjecture can be made by comparing SDE \eqref{FP1} in Example \ref{Ex FP} and SDE \eqref{SDE81}, 
     both for the particular $E$ discussed in the above paragraph which is assumed independent of $B$.  
     First, SDE \eqref{SDE81} is particularly different from SDE \eqref{FP1} due to the presence of the $dt$ term.  
     Second, Theorem 4.1 in Hahn, Kobayashi and Umarov~\cite{HKU} shows that the transition probability density of the solution to SDE 
     \eqref{FP1} satisfies PDE \eqref{FP2}, 
     and the proof is carried out by taking the expectation in the time-changed It\^o formula \eqref{ITO3}.  
     Consequently, \eqref{SDE87} suggests that if SDE \eqref{FP1} is replaced by an SDE having a term $\rho(X_t) dt$, 
     then the corresponding PDE may involve a fractional integral term.  \qed
\end{em}
\end{Ex}

     The following two examples clarify how to apply the reduction method obtained in Proposition \ref{Prop Reduction}.

\begin{Ex}\label{Ex LSDE-9}\par
\begin{em}
     Solution \eqref{SDE42} to the homogeneous linear SDE \eqref{SDE41} discussed in Example \ref{Ex LSDE-4} 
     can also be obtained by using the technique provided in Proposition \ref{Prop Reduction}. 
     In this case, the integrating factor is $U_t=\exp\bigl\{(\sigma^2/2-\mu)E_t-\sigma B_{E_t}\bigr\}$     
     and \eqref{SDE94} becomes the path-by-path ODE $dW_t=\rho W_tdt$ with $W_0=x_0$, 
     which has the solution $W_t=x_0\, e^{\rho t}$.  Hence, the relation $X_t=U_t^{-1}W_t$ immediately 
     yields the desired solution form \eqref{SDE42}. 
     More generally, the same reduction scheme proves Proposition \ref{Prop LSDE-2}.  \qed  
\end{em}
\end{Ex}

\begin{Ex}\label{Ex LSDE-10}\par
\begin{em}
     As another application of the reduction method introduced in Proposition \ref{Prop Reduction}, 
     consider a generalized population growth model 
     \begin{align}
          dX_t=q X_t (K-X_t) dt + \mu X_t dE_t + \sigma X_t dB_{E_t} \ \ \  \textrm{with} \ X_0=x_0 \label{SDE101}
     \end{align}
     where $q$, $K$, $x_0>0$ and $\mu$, $\sigma\in\bbR$.
     This model describes the growth of a population of size $X_t$ in some environment.  
     $q$ and $K$ represent the quality and the carrying capacity of the environment, respectively.  
     If the quality of life is good and the current population is less than the carrying capacity, i.e., if $q$ is large and $0<X_t<K$, 
     then the population will grow, i.e., the drift coefficient $q X_t(K-X_t)$ is positive.  On the other hand, 
     a population exceeding the capacity of the environment is expected to decrease even when the quality is good, 
     i.e., if $X_t>K$, then the drift $q X_t(K-X_t)$ is negative, regardless of the value of $q(>0)$.  
 
     Note that SDE \eqref{SDE101} possesses a distinct form of coefficients in $dt$ and $dE_t$ terms, 
     unlike Examples \ref{Ex LSDE-4}, \ref{Ex LSDE-6} and \ref{Ex LSDE-7}.   
     Hence, this model is constructed based on Viewpoint 2 of Remark \ref{Rems LSDE-solution} (b). 
     The presence of the term $\mu X_t dE_t$ implies that a certain factor 
     originating in the new clock affects the growth of the population,  
     and the effect is proportional to the current position $X_t$.   
     $\sigma$ describes the noise of the system as in the classical population growth model 
     (i.e., SDE \eqref{SDE101} with $E_t=t$ and $\mu=0$). 

     Theorem \ref{Thm LSDE-3} cannot be applied to the nonlinear SDE \eqref{SDE101}. 
     Instead, Proposition \ref{Prop Reduction} with $W_t=U_tX_t$ 
     where $U_t=\exp\bigl\{(\sigma^2/2-\mu)E_t-\sigma B_{E_t}\bigr\}$, yields the path-by-path ODE
     \begin{align}
          \frac{dW_t}{dt}=q W_t \bigl(K-U_t^{-1}W_t\bigr) \ \ \ \textrm{with} \ W_0=x_0. \label{SDE102}
     \end{align}
     Consider a Bernoulli-type ODE 
     \begin{align}
          y'(t)=f(t)y^2(t)+k y(t) \ \ \ \textrm{with} \ y(0)=x_0, \label{SDE103}
     \end{align}
     where $k$ is a real constant and the symbol $'$ denotes the derivative with respect to $t$.  
     By the substitution $z(t)=y^{-1}(t)$, the ODE \eqref{SDE103} reduces to 
     $z'(t)+kz(t)=-f(t)$ with $z(0)=x_0^{-1}$.  Multiplication of both sides by $e^{kt}$ leads to 
     $\bigl\{ e^{kt}z(t)\bigr\}'=-e^{kt}f(t)$, whose solution is
     \begin{align}
          e^{kt}z(t)-x_0^{-1}=-\int_0^t e^{ks}f(s)ds, \ \ \ \textrm{or} \ \ \ y(t)=\dfrac{e^{kt}}{x_0^{-1}-\int_0^t e^{ks}f(s)ds}. \label{SDE104}
     \end{align}
     By the substitutions, $y(t)=W_t$, $f(t)=-q U_t^{-1}$, $k=q K$ in \eqref{SDE104}, 
     \begin{align}
          X_t&=U_t^{-1}W_t
          = \dfrac{U_t^{-1}\cdot \exp \{q Kt\} } {x_0^{-1}+\int_0^t \exp \{q Ks\} \cdot q U_s^{-1} ds}\label{SDE105} \\
          &= \dfrac{\exp\bigl\{q Kt+(\mu- \frac 12 \sigma^2)E_t+\sigma B_{E_t} \bigr\}} 
                      {x_0^{-1}+q \int_0^t \exp\bigr\{q Ks+(\mu-\frac 12 \sigma^2)E_s+\sigma B_{E_s}\bigr\} ds}\;, \notag
     \end{align}  
     yielding the solution to the generalized population growth model \eqref{SDE101}. \qed 
\end{em}
\end{Ex}

\appendix

\section*{Appendix --- Construction of Stochastic Integrals}

     The aim of this appendix is to make explicit the class ${L(Z,{\cal F}_t)}$ of 
     $Z$-integrable predictable processes treated in this paper.      
     For details regarding the construction of stochastic integrals driven by a semimartingale, consult \cite[II--IV]{Protter}.  

     Throughout, a filtration $\cF$ satisfying the usual conditions is fixed.  
     Write $\bbD=\bbD\cF$ (c\`adl\`ag adapted processes), $\bbL=\bbL\cF$ (c\`agl\`ad adapted processes),
     and $\mathcal{P}=\mathcal{P}\cF$ (predictable processes).     
     Let $\bddL$ and $\bddP$ denote bounded processes in the specified class.  
     Let $\bbS$ be a subset of $\bbL$ consisting of all processes of the form 
     $H_t=H_0 \, \bI_{\{0\}}(t)+\sum_{i=1}^n H_i \, \bI_{(T_i,T_{i+1}]}(t)$, 
     where $n$ is a positive integer, $\{T_i\}_{i=1}^{n+1}$ is an increasing sequence of finite stopping times with $T_1=0$, and 
     each $H_i$ is an $\mathcal{F}_{T_i}$-measurable random variable.  
     
     First, endow $\bbD$, $\bbL$ and $\bbS$ with the topology induced by ``$H^{m}\longrightarrow H$ if and only if for each $t\geq 0$, 
     $\sup_{0\leq s\leq t}|H_s^m-H_s|\longrightarrow 0$ in probability as $m\to\infty$.''  Then $\bbS$ is a dense subspace of $\bbL$,
     and $\bbD$ becomes a complete metric space with a compatible metric 
     $d(Y,Z)\DEF \sum_{n=1}^\infty (1/2^n) \bbE\bigl[\min\bigl(1,\sup_{0\leq s\leq t}|Y_s-Z_s|\bigr)\bigr]$.  
     Given a semimartingale $Z$ starting at $0$, the stochastic integral of $H\in\bbS$ of the above form is defined to be 
     $H\bullet Z\DEF J_{Z}(H)\DEF \sum_{i=1}^n H_i(Z^{T_{i+1}}-Z^{T_i})$ where $Z_t^T\DEF Z_{\min(t, T)}$.  
     The continuous linear operator $J_Z:\bbS\longrightarrow \bbD$ uniquely extends to an operator defined on $\bbL$. 
     For the moment, denote $J_{Z}(H)$ as $[\bD 1\textrm{-} ]H\bullet Z$ for $H\in \bbL$.  
     Note that the quadratic variation of $Z$ is defined by \eqref{QUAD} via this integral operator.  

     The next step is to introduce the space $\HH$ of semimartingales starting at $0$ 
     with a unique decomposition $\widetilde{Z}=\widetilde{M}+\widetilde{A}$ 
     where $\widetilde{M}$ is a local martingale and $\widetilde{A}$ is a predictable process of finite variation such that 
     \begin{align*}
     \bigl\|\widetilde{Z}\bigr\|_{\HH}
     \DEF \bigl\|[\widetilde{M},\widetilde{M}]_\infty^{1/2}\bigr\|_{L^2}+\Bigl\|\int_0^\infty |d\widetilde{A}_s|\hspace{0.5mm}\Bigr\|_{L^2}<\infty.
     \end{align*}   
     The real vector space $\HH$ with the norm $\|\cdot \|_{\HH}$ forms a Banach space.  
     To extend a class of integrands, first fix an integrator $\widetilde{Z}=\widetilde{M}+\widetilde{A}\in \HH$ 
     and introduce a metric $d_{\widetilde{Z}}$ on $\bddP$ by 
     \begin{align*}
          d_{\widetilde{Z}}(H,K)\DEF \Bigl\| \Bigl\{\int_0^\infty (H_s-K_s)^2 d[\widetilde{M},\widetilde{M}]_s\Bigr\}^{1/2}\Big\|_{L^2}
                                        +\Big\|\int_0^\infty |H_s-K_s| |d\widetilde{A}_s|\hspace{0.5mm}\Big\|_{L^2}
     \end{align*}
     where $|d\widetilde{A}_s|$ denotes the integral with respect to the total variation measure.   
     The integrals appearing in this definition are understood path-by-path in the Lebesgue--Stieltjes sense, and it follows that
     $d_{\widetilde{Z}}(H,K)=\bigl\| H\bullet\widetilde{Z}-K\bullet\widetilde{Z}\bigr\|_{\HH}$.  
     Under this metric, $\bddL$ is dense in $\bddP$.  For $H\in \bddP$, it is easy to see that a unique $\HH$-limit 
     of the sequence $\{[\bD 1\textrm{-} ]H^n\bullet \widetilde{Z}\}$ exists where $\{H^n\}$ is an approximating sequence in $\bddL$ for $H$.  
     Moreover, the limit is determined independently of the choice of the approximating sequence.  
     Hence, the stochastic integral 
     $[\bD 2\textrm{-}]H\bullet  \widetilde{Z}\DEF 
     \HH\textrm{-}\lim_{n\to\infty} [\bD 1\textrm{-} ]H^n\bullet \widetilde{Z}$ is well-defined.  

     The third step requires another class of integrands, denoted $L_{\HH}(\widetilde{Z},{\cal F}_t)$, 
     which consists of predictable processes with 
     \begin{align*}
     \Bigl\| \Bigl\{\int_0^\infty H_s^2 d[\widetilde{M},\widetilde{M}]_s\Bigr\}^{1/2} \Bigr\|_{L^2}
     +\Bigl\|\int_0^\infty |H_s||d\widetilde{A}_s|\hspace{0.5mm}\Bigr\|_{L^2}<\infty,
     \end{align*} 
     where $\widetilde{Z}=\widetilde{M}+\widetilde{A}\in \HH$.  Associate to $H\in L_{\HH}(\widetilde{Z},{\cal F}_t)$, 
     the truncation processes $\{H^k\}$ in $\bddP$, given by $H^k\DEF H\, \bI_{\{|H|\leq k\}}$. 
     Again, via the same reasoning as above, the stochastic integral $[\bD 3\textrm{-} ]H\bullet \widetilde{Z}$ is defined to be 
     the unique $\HH$-limit of the sequence $\{[\bD 2\textrm{-} ]H^k\bullet \widetilde{Z}\}$.  That is,  
     $[\bD 3\textrm{-}]H\bullet \widetilde{Z}\DEF \HH\textrm{-}\lim_{n\to\infty} [\bD 2\textrm{-} ]H^n\bullet \widetilde{Z}$. 
    
     Finally, given a general semimartingale $Z$ starting at $0$, a predictable process $H$ 
     is said to be $Z$\textit{-integrable}, denoted $H\in {L(Z,{\cal F}_t)}$, 
     if there exists a sequence $\{\sigma^n\}$ of stopping times increasing to $\infty$ such that $\widetilde{Z}^n\DEF Z^{\sigma^n-}\in \HH$ 
     and $H\in L_{\HH}(\widetilde{Z}^n,{\cal F}_t)$ for each $n$, where $Z_t^{\sigma-}\DEF Z_t \, \bI_{[0,\sigma)}(t)+Z_{\sigma-} 
     \, \bI_{[\sigma,\infty)}(t)$.   
     With this sequence $\{\sigma^n\}$, the stochastic integral of $H$ driven by $Z$ is defined to be 
     $H\bullet Z\DEF [\bD 3\textrm{-} ]H\bullet \widetilde{Z}^n$ on $[0,\sigma^n)$.  
     This definition is consistent and independent of the choice of the localizing sequence $\{\sigma^n\}$.  

     One important special case is when $Z=M$ is a continuous $\cF$-local martingale.  
     In this case, $H\in L(M,\mathcal{F}_t)$ if and only if $H\in \mathcal{P}\cF$ 
     and $\bbP\bigl( \int_0^t H_s^2 d[M,M]_s< \infty \bigr)=1$ for all $t\geq 0$.  Moreover, the stochastic integral $H\bullet M$ 
     is also a continuous $\cF$-local martingale.  In particular, if $Z=B$ is a standard $\cF$-Brownian motion 
     and $E$ is a continuous $\cF$-time-change, 
     then it is easily shown that $(B_{E_t})$ is a continuous $\cG$-local martingale, where $\mathcal{G}_t\DEF \mathcal{F}_{E_t}$.  
     Thus, for any $K\in L(B\circ E,\mathcal{G}_t)$, the stochastic integral $K\bullet (B\circ E)$ 
     is also a continuous $\cG$-local martingale.  

\section*{Acknowledgements}
     I am grateful to Marjorie Hahn and Sabir Umarov 
     for their direction and assistance throughout this research.  
     I also wish to thank 
     Meredith Burr, Jamison Wolf, and Xinxin Jiang 
     for comments and suggestions.

\end{document}